\documentclass[letter, reqno]{amsart}

\usepackage[usenames,dvipsnames]{color}
\usepackage[utf8]{inputenc}
\usepackage{amsthm,amsfonts,amssymb,amsmath,amsxtra}
\usepackage[all]{xy}
\SelectTips{cm}{}

\usepackage{cite}

\usepackage{verbatim}
\usepackage{xr-hyper}
\usepackage{varioref}
\usepackage{hyperref}
\usepackage[nameinlink]{cleveref}
\usepackage{tikz-cd}
\usepackage[margin=1.25in]{geometry}
\usepackage{mathrsfs}
\usepackage{xcolor}

\usepackage[shortlabels]{enumitem}
\usepackage{lipsum}
\usetikzlibrary{positioning,chains,fit,shapes,calc,arrows,patterns,external,shapes.callouts,graphs}

\RequirePackage{xspace}
\RequirePackage{etoolbox}
\RequirePackage{varwidth}
\RequirePackage{enumitem}
\RequirePackage{tensor}
\RequirePackage{mathtools}
\RequirePackage{longtable}
\RequirePackage{multirow}

\def\ge{\geqslant}
\def\le{\leqslant}
\def\a{\alpha}
\def\b{\beta}
\def\g{\gamma}
\def\G{\Gamma}

\def\D{\Delta}

\def\e{\epsilon}

\def\s{\sigma}

\def\k{\kappa}
\def\l{\lambda}
\def\z{\zeta}

\def\i{^{-1}}

\def\<{\langle}
\def\>{\rangle}

\newcommand{\ba}{\mathbf a}

\newcommand{\bG}{\mathbf G}

\newcommand{{\BG}}{\ensuremath{\mathbb {G}}\xspace}

\newcommand{{\BK}}{\ensuremath{\mathbb {K}}\xspace}

\newcommand{\BQ}{\ensuremath{\mathbb {Q}}\xspace}
\newcommand{\BR}{\ensuremath{\mathbb {R}}\xspace}
\newcommand{\BS}{\ensuremath{\mathbb {S}}\xspace}

\newcommand{\BZ}{\ensuremath{\mathbb {Z}}\xspace}

\newcommand{\CA}{\ensuremath{\mathcal {A}}\xspace}

\newcommand{\CC}{\ensuremath{\mathcal {C}}\xspace}

\newcommand{\CE}{\ensuremath{\mathcal {E}}\xspace}

\newcommand{\CI}{\ensuremath{\mathcal {I}}\xspace}

\newcommand{\CM}{\ensuremath{\mathcal {M}}\xspace}
\newcommand{\CN}{\ensuremath{\mathcal {N}}\xspace}
\newcommand{\CO}{\ensuremath{\mathcal {O}}\xspace}

\newcommand{\CS}{\ensuremath{\mathcal {S}}\xspace}

\newcommand{\wt}{\text{wt}}

\DeclareMathOperator{\Adm}{Adm}

\DeclareMathOperator{\rank}{rank}

\newcommand{\wtd}{\widetilde}

\def\tS{\tilde \BS}

\newtheorem*{thrm1}{Theorem A}

\newtheorem*{thrm2}{Theorem B}

\newtheorem*{thrm3}{Theorem C}

\newtheorem{theorem}{Theorem}
\newtheorem{proposition}[theorem]{Proposition}
\newtheorem{lemma}[theorem]{Lemma}

\newtheorem{corollary}[theorem]{Corollary}

\theoremstyle{definition}
\newtheorem*{acknowledgement}{Acknowledgement}
\newtheorem{definition}[theorem]{Definition}

\newtheorem{remark}[theorem]{Remark}

\numberwithin{equation}{section}
\numberwithin{theorem}{section}

\setitemize[0]{leftmargin=*,itemsep=\the\smallskipamount}
\setenumerate[0]{leftmargin=*,itemsep=\the\smallskipamount}

\renewcommand{\to}{%
   \ifbool{@display}{\longrightarrow}{\rightarrow}%
   }
\let\shortmapsto\mapsto
\renewcommand{\mapsto}{%
   \ifbool{@display}{\longmapsto}{\shortmapsto}%
   }
\newlength{\olen}
\newlength{\ulen}
\newlength{\xlen}
\newcommand{\xra}[2][]{%
   \ifbool{@display}%
      {\settowidth{\olen}{$\overset{#2}{\longrightarrow}$}%
       \settowidth{\ulen}{$\underset{#1}{\longrightarrow}$}%
       \settowidth{\xlen}{$\xrightarrow[#1]{#2}$}%
       \ifdimgreater{\olen}{\xlen}%
          {\underset{#1}{\overset{#2}{\longrightarrow}}}%
          {\ifdimgreater{\ulen}{\xlen}%
             {\underset{#1}{\overset{#2}{\longrightarrow}}}
             {\xrightarrow[#1]{#2}}}}%
      {\xrightarrow[#1]{#2}}
   }
\makeatother
\newcommand{\xyra}[2][]{%
   \settowidth{\xlen}{$\xrightarrow[#1]{#2}$}%
   \ifbool{@display}%
      {\settowidth{\olen}{$\overset{#2}{\longrightarrow}$}%
       \settowidth{\ulen}{$\underset{#1}{\longrightarrow}$}%
       \ifdimgreater{\olen}{\xlen}%
          {\mathrel{\xymatrix@M=.12ex@C=3.2ex{\ar[r]^-{#2}_-{#1} &}}}%
          {\ifdimgreater{\ulen}{\xlen}%
             {\mathrel{\xymatrix@M=.12ex@C=3.2ex{\ar[r]^-{#2}_-{#1} &}}}
             {\mathrel{\xymatrix@M=.12ex@C=\the\xlen{\ar[r]^-{#2}_-{#1} &}}}}}%
      {\mathrel{\xymatrix@M=.12ex@C=\the\xlen{\ar[r]^-{#2}_-{#1} &}}}%
   }
\makeatletter
\newcommand{\xla}[2][]{%
   \ifbool{@display}%
      {\settowidth{\olen}{$\overset{#2}{\longleftarrow}$}%
       \settowidth{\ulen}{$\underset{#1}{\longleftarrow}$}%
       \settowidth{\xlen}{$\xleftarrow[#1]{#2}$}%
       \ifdimgreater{\olen}{\xlen}%
          {\underset{#1}{\overset{#2}{\longleftarrow}}}%
          {\ifdimgreater{\ulen}{\xlen}%
             {\underset{#1}{\overset{#2}{\longleftarrow}}}
             {\xleftarrow[#1]{#2}}}}%
      {\xleftarrow[#1]{#2}}
   }
\newcommand{\isoarrow}{%
   \ifbool{@display}{\overset{\sim}{\longrightarrow}}{\xrightarrow\sim}%
   }
   
\begin{document}

\title[Affine Deligne-Lusztig Varieties and Quantum Bruhat Graph]{Affine Deligne-Lusztig Varieties and Quantum Bruhat Graph}

\author{Arghya Sadhukhan}
\address[A. S.]{Department of Mathematics, University of Maryland, College Park, MD 20742}

\keywords{Affine Deligne-Lusztig variety, generic Newton point, affine Weyl group, dimension formula}
\email{arghyas0@math.umd.edu}
\subjclass[2010]{20G25,11G25,20F55}
\thanks{}

\date{\today}

\begin{abstract}
In this paper, we consider affine Deligne-Lusztig varieties $X_w(b)$ and their certain union $X(\mu,b)$ inside the affine flag variety of a reductive group. Several important results in the study of affine Deligne-Lusztig varieties have been established under the so-called superregularity hypothesis. Such results include a description of generic Newton points in Iwahori double cosets of loop groups, covering relation in associated Iwahori-Weyl group and dimension formula for $X(\mu,b)$. We show that one can considerably weaken the superregularity hypothesis and sometimes completely eliminate it, thus strengthening these existing results.
\end{abstract}
\maketitle
\section{Introduction}
\subsection{Motivation}
Let $F$ be a nonarchimedean local field and $\breve F$ be the completion of its maximal unramified extension. Let $\s$ be the Frobenius automorphism of $\breve F$ over $F$. Let $\bG$ be a split connected reductive group over $F$. Let $\breve \CI$ be the standard Iwahori subgroup of $\bG(\breve F)$. Let $w$ be an element in the Iwahori-Weyl group $\wtd W$ and $[b]$ be a $\s$-conjugacy class in $G(\breve F)$. The affine Deligne-Lusztig variety associated to this pair $(w,[b])$ is defined to be a locally closed subscheme of the affine flag variety given by
\begin{gather*}
X_w(b)=\{g \breve \CI \in \bG(\breve F)/\breve \CI; g \i b \s(g) \in \breve \CI \dot w \breve \CI\}.
\end{gather*}

Let $\mu$ be a conjugacy class of coweights of $\breve G$ over the algebraic closure $\bar{F}$ and let $\Adm(\mu)$ be the set of $\mu$-admissible elements of $\wtd W$; see \cref{sec:ADLV} for more details. Then we define
\[X(\mu,b)=\coprod\limits_{w \in \Adm(\mu)} X_w(b).\] This is a closed subscheme of the affine flag variety and it serves as the group-theoretic model for the Newton stratum corresponding to $[b]$ in the special fiber of a Shimura variety giving rise to the
datum $(\bG,\mu)$.

First introduced by Rapoport in \cite{Rapoport05}, the notion of affine Deligne-Lusztig varieties plays an important role in arithmetic geometry and the Langlands program. A fundamental problem in the study of $X_w(b)$ is the non-emptiness pattern, i.e. for which choice of $(w,[b])$, $X_w(b)$ is nonempty.  Note that $X_w(b) \neq \emptyset$ if and only if $\CN_{b,w}:=[b]\cap \breve \CI w \breve \CI$ is nonempty; the latter is the set of geometric points of a locally closed subscheme of $\breve \CI w \breve \CI$; namely, the Newton stratum associated to $[b]$.

Let $B(G)$ be the set of $\s$-conjugacy classes of $\bG(\breve F)$. The set $B(G)$ is equipped with a
natural partial order defined in terms of the Kottwitz map $\k$ and the Newton map $\nu$. Understanding the non-emptiness pattern of affine Deligne Lusztig varieties associated to $w \in \wtd W$ naturally lead us to investigate the sub-poset $B(G)_w = \{[b]: X_w(b)\neq \emptyset\}$. This set contains a unique maximal element, which we denote by $[b_w]$. Recent work by Mili\'cevi\'c and Viehmann in \cite{VM20} show that whenever the dimension of $X_w(b_w)$ agrees with its virtual dimension in the sense of \cite{He14}, the associated Newton strata inside $\breve \CI w \breve \CI$ exhibit well-behaved geometry. Furthermore, their theorem gives a condition that can be checked from knowledge of the Newton point $\nu([b_w])$ for this maximal element of the $B(G)_w$, but it provides important conclusion about the shape of the entire poset. In light of these results, it becomes important to understand the maximal Newton point $\nu([b_w])$ for elements $w \in \wtd W$.

An explicit description of $\nu([b_w])$ is given by Mili\'cevi\'c in \cite[Theorem 3.2]{Milicevic21} for elements $w \in \wtd W$ that are \emph{suitably far from walls of any Weyl chamber}. To quantify the later condition, let us define for any dominant coweight $\l$ its depth as \[\text{depth}(\l) = \min \{\<\a,\l\>:\a~\text{is a simple root}\}.\]
This quantity estimates how far $\l$ is from the walls of any Weyl chamber. Then a formula for $\nu([b_w])$ obtained in loc. sit. under `superregularity' hypothesis on the dominant translation part $\l$ of $w$, which roughly says that the depth of $\l$ is \emph{quadratically large} with respect to the semisimple rank of $\bG$. To establish this result, Mili\'cevi\'c first derives a characterization of the covering relation in $\wtd W$, again under certain superregularity hypothesis. 

This characterization is of independent interest, and more recently it has been utilized by He and Yu in \cite{HeYu20} to deduce a partial description of admissible sets in $\wtd W$. This latter description is a crucial ingredient in proving the main result in their paper, which provides a formula for the dimension of $X(\mu,b)$. However, due to this reliance on the result about covering relation, such dimension formula is established in loc. sit. under the hypothesis that $\mu$ is superregular.

\subsection{Main results}
The main purpose of this article is to show that the aforementioned results in the study of affine Deligne-Lusztig varieties can be improved by weakening the hypothesis required to prove them.
\subsubsection{} An uniform bound is given in \cite[Corollary 3.3]{Milicevic21} in the quasi-simple case, which says that the description of $\nu([b_w])$ provided in the article is valid whenever the following depth hypothesis is satisfied on $\l$:
\begin{equation*}
depth(\lambda) \geq 
\begin{cases}
8 \ell(w_0), & \text{if $\bG~ \text{is of classical type}$;}\\
16 \ell(w_0), & \text{if $\bG~ \text{is of exceptional type}$.}\\
\end{cases}
\end{equation*}
Here $\ell(w_0)$ is the length of the longest element $w_0$ in the associated finite Weyl group $W$. However, it is noted in loc. sit. that this restriction on the lower bound arises merely from the proof method and is rather superficial, in the sense that depending on the element $w$ one can often get the same formula for $\nu([b_w])$ under a much weaker hypothesis. In this paper, we show that one can indeed weaken this hypothesis considerably. For simplicity, we focus on the quasi-simple groups in the statement of this result.
\begin{thrm1}[\Cref{Generic-nu}]
Suppose that $\bG$ is a quasi-simple split group of rank $n$. Let $w=ut^\l v$ be an element of its Iwahori-Weyl group $\wtd W$ such that $\text{depth}(\l) > \Xi $, where $\Xi$ is a linear expression of $n$. Then the maximal Newton point associated to $w$ is given by $\nu([b]_w)=\l-\text{weight}(v^{-1},u)$.
\end{thrm1}
Here the precise expression of $\Xi$ depends on the group $\bG$. We refer to \cref{sec:QBG} for the definition of \emph{weight} and to the statement of \cref{Generic-nu} and its subsequent remark for the explicit form of $\Xi$ and relevant discussion on how to handle the case when $\bG$ is a split connected reductive group which is not necessarily quasi-simple.
\subsubsection{} Our second main result characterizes the covering relation of Bruhat order for most of the elements in $\wtd W$. 
\begin{thrm2}[\Cref{cover}]
Assume that $w=ut^\l v$ be an element of $\wtd W$ such that $\text{depth}(\l)$ is bigger than a certain constant. Let $r_\b=t^{mu\a^\vee}s_{u\a}$ be an affine reflection for some positive root $\a$ and integer $m$, and let $w'=r_\b w$. Then $w \gtrdot w'$ is a covering relation, i.e. $w \geq w'$ and $\ell(w)=\ell(w')+1$, if and only if one of the following conditions holds:
\begin{enumerate}
    \item $m=0$ and $\ell(us_\a)=\ell(u)-1$; in this case, $w'=us_\a t^\l v$.
    \item $m=1$ and $\ell(us_\a)=\ell(u)+\<2\rho,\a^\vee\>-1$; in this case, $w'=us_\a t^{\l-\a^\vee} v$. 
    \item $m=\<\a,\l\>$ and $\ell(s_\a v)=\ell(v)+1$; in this case, $w'=ut^\l s_\a v$.
    \item $m=\<\a,\l\>-1$ and $\ell(s_\a v)=\ell(v)-\<2\rho,\a^\vee\>+1$; in this case, $w'=ut^{\l-\a^\vee}s_\a v$.
\end{enumerate} 
\end{thrm2}
We refer to the statement of \Cref{cover} and its subsequent remark for the precise constant lower bound mentioned above. We remark that the same characterization was first proved in \cite[Proposition 4.1]{LS10} under the assumption that the depth of the relevant coweight is at least $2|W|+2$. Later it was improved in \cite[Proposition 4.2]{Milicevic21}, where the depth hypothesis is relaxed to a quadratic lower bound in the quasi-simple case as follows: 
\begin{equation*}
depth(\lambda) \geq 
\begin{cases}
2 \ell(w_0)+2, & \text{if $\bG~ \text{is not of type } G_2$;}\\
3 \ell(w_0)+3, & \text{if $\bG~ \text{is of type } G_2$.}\\
\end{cases}
\end{equation*}

\subsubsection{} In \cref{sec:dim}, we improve upon a result about dimension of $X(\mu,b)$. Based on recent discovery of the dimension formula for single affine Deligne-Lusztig variety $X_w(b)$ for ``sufficiently large" $w$ in \cite{He20+}, He and Yu provide a formula in \cite{HeYu20} for the dimension of $X(\mu,b)$ under the assumption that $\mu$ is superregular. We show that this formula in fact holds true for all \emph{dominant regular $\mu$}. 
\begin{thrm3}
Suppose that $\textbf{G}$ is a split connected reductive group. Let $\mu \in X_*(T)$ be dominant regular. Assume that $[b] \in B(G,\mu)$ with $\mu \geq \nu([b])+2\rho^\vee$. Then 
\begin{equation*}
    \dim X(\mu,b)=\<\rho, \mu-\nu([b])\>-\frac{1}{2}\text{def}_{\textbf{G}}(b)+\frac{1}{2}(\ell(w_0)-\ell_R(w_0)).
\end{equation*}
\end{thrm3}

We refer the reader to \cref{sec:ADLV} and \cref{sec:virdim} for relevant definitions. 

\subsection{Strategy}
We now discuss the outline of proofs of the above results. We refer to \cref{sec:setting} for relevant definitions.

Let us start by summarising our approach to prove theorem A. By employing certain tools from the theory of Demazure product, we show in \cref{sec:reduction} that one can recast the problem of computing the maximal Newton point associated with an arbitrary element of $\wtd W$ to that of an element lying in the dominant chamber. Hence, we focus exclusively on such element $w$ . To do so, we repeatedly apply an observation about Bruhat order on $\wtd W$ from \cite{Rapoport05} until we reach a dominant translation element below $w$. In order to carry out this procedure, we impose a certain lower bound on the depth of the coweight associated to $w$ in \cref{sec:bd1}. We also show that under this depth hypothesis, the Newton point of $[b_w]$ is given by a translation element below it. We finally confirm that this is indeed the translation element produced in the process alluded to before. This is achieved by proving a technical lemma, which employs techniques about downward Demazure product, and as such brings in another restriction in the depth hypothesis. Taking into account these two restrictions on the depth, we finally prove the theorem by appealing to the fact that the result is already known under superregularity hypothesis in \cite{Milicevic21}. 

\smallskip
To prove theorem B, we again (easily) reduce to the case of considering elements lying in the dominant chamber. Then we estimate a series of relevant quantities to assert that the coweight associated with the dominant translation part of the plausible cocover has to remain in the closure of one of the two specific Weyl chambers in order for covering relation to hold. Finally, we deal with each of these two cases separately. We note that a crucial aspect in which our proof differs from that in \cite{Milicevic21} is that we handle the cases where this coweight can be possibly singular. 

\smallskip
We remark that in the treatment of \cite{Milicevic21}, the covering relation in $\wtd W$ is in fact the starting point towards obtaining the formula of maximal Newton point. The strategy in loc. sit. is to apply the characterization of covering relation on $w$ and its subsequent cocovers repeatedly, until one reaches the maximal translation element below $w$. One then carefully keeps track of the cocovers of a given element $w$ to quantify a depth hypothesis under which this process work. Hence, it is plausible that one may use an improved version of the covering relation characterization (as the one established in theorem B) and carry out a similar analysis to obtain the maximal Newton point formula under a weaker hypothesis on the depth, \emph{which may even be a constant lower bound hypothesis}. However, we do not follow this approach here. While we prove theorem A only under a linear bound hypothesis on depth, our approach exclusively outlines techniques from Demazure product and \emph{it highlights the fact that the simple observation in \cref{Base case} lies at the heart of this matter}. We do point out some speculative assumptions about how one can improve theorem A, see \cref{speculation}.

\smallskip
Let us now briefly discuss our strategy for proving theorem C. The key ingredient in eliminating superregularity condition here is the additivity property of admissible sets established in \cite[Theorem 5.1]{He16}. This is one of the crucial results used in proving Kottwitz-Rapoport conjecture on non-emptiness of $X(\mu,b)$ in loc. sit. We use this additivity to show that certain steps in the existing proof in \cite{HeYu20} can be carried out in a way that superregularity is no longer required. We point out that the proof philosophy has a resemblance to that of theorem A, in the sense that we leverage the results already established in loc. sit. under superregularity hypothesis. 

\begin{acknowledgement}
I am grateful to my advisor Xuhua He for the suggestion to investigate these problems as well as numerous helpful discussions over the course of this work. I would also like to thank Jeffrey Adams and Thomas Haines for many helpful conversations and valuable feedback. During the preparation of this article, I was partially supported by a Graduate School Summer Research Fellowship and NSF grant DMS no. 1801352 through Thomas Haines. Upon completion of writing this paper, I was informed that Felix Schremmer had independently obtained similar results. I thank him for his thoughtful comments on a preliminary draft of the paper, especially for pointing out a minor error in the manuscript.
\end{acknowledgement}

\section{Preliminary}\label{sec:setting}

\subsection{Notations}\label{sec:notation} Recall that $F$ is a nonarchimedean local field and $\breve F$ is the completion of the maximal unramified extension of $F$. Let $\bG$ be a split connected reductive group over $F$; We write $\breve G$ for $\bG(\breve F)$. Let $\s$ be the Frobenius morphism of $\breve F/F$. We use the same symbol $\s$ for the induced Frobenius morphism on $\breve G$. Let $B$ be a fixed Borel subgroup with $T$ a split maximal torus in $B$.

Let $\CA$ be the apartment of $\bG$ corresponding to $T$. We fix an alcove $\ba$ in $\CA$, and let $\breve \CI \subset \breve G$ be the Iwahori subgroup corresponding to $\ba$. Then $\breve \CI$ is $\s$-stable. We denote by $N$ the normalizer of $T$ in $\bG$. The \emph{Iwahori-Weyl group} (associated to $T$) is defined as \[\wtd W= N(\breve F)/T(\breve F) \cap \breve \CI.\] 

Let $W=N(\breve F)/T(\breve F)$ be the finite Weyl group. Fixing a special vertex of the base alcove $\ba$, we have the splitting $$\wtd W=X_*(T) \rtimes W=\{t^{\l} x: \l \in X_*(T), x \in W\}.$$

For any $w \in \wtd W$, we choose a representative in $N(\breve F)$ which is also denoted by $w$. Note that the action $\s$ on $\breve G$ induces a natural action of $\s$ on $\wtd W$, which is just the identity automorphism since $\bG$ is split. We remark that one can identify the element $x \in \widetilde{W}$ with $x \ba$, the (extended) alcove that one obtains as image of the base alcove $\ba$ under $x$. 

Denote by $\tS$ the set of simple reflections in $\wtd W$; these are taken to be reflections along the walls of $\ba$. Let $\BS \subset \tS$ be the set of simple reflections in $W$. Let $\Phi$ (resp. $\Phi^+,\D$) be the set of all roots (resp. positive roots, simple roots). Let $w_0$ be the longest element in $W$. Let $\rho$ be the dominant weight with $\<\a^\vee, \rho\>=1$ for any $\a \in \D$ and $\rho^\vee$ be the dominant coweight with $\<\rho^\vee, \a\>=1$ for any $\a \in \D$. Let $\theta$ be the highest root. We denote the dominant (rational) Weyl chamber by $\CC^+ \subset X_*(T)_{\BQ}$, and the image of this chamber under some element $x \in W\setminus \{1\}$ by $\CC_x$. For an irreducible Weyl group $W$, we follow the labeling of roots as in \cite{Bou} and we usually write $s_i$ instead of $s_{\a_i}$, where $\D=\{\a_i: 1 \leq i \leq n\}$. For any $J \subset \tilde \BS$, we denote by $W_J$ the subgroup of $\wtd W$ generated by to $J$. Let ${}^J \wtd W$ (resp. $\wtd W^J$) be the set of minimal length elements in their cosets in $W_J \backslash \wtd W$ (resp. $\wtd W/W_J$).

\subsection{Length function and Bruhat order on $W$ and $\wtd W$}\label{sec: l and <}
We denote by $\ell$ the length function on $\wtd W$ determined by the base alcove $\ba$. Let $W_a$ be the subgroup of $\wtd W$
generated by $\tS$. Then $W_a$ is an affine Weyl group. Let $\Omega \subset \wtd W$ be the subgroup of length-zero elements (or equivalently, the stabilizer of $\ba$) in $\wtd W$. Then $$\wtd W = W_a \rtimes \Omega.$$
$W_a$ is the Coxeter group associated to $\tS$, and hence it comes equipped with an associated Bruhat order, which we denote by $\leq$. This is extended to $\wtd W$ as follows. For two elements $\wtd w_1, \wtd w_2 \in \wtd W$, use the above decomposition to write $\wtd w_i=w_i\z_i$ with $w_i \in W_a, \z_i \in \Omega$ for $i=1,2$. We then declare $\wtd w_1 \leq \wtd w_2$ if $w_1 \leq w_2$ and $\z_1=\z_2$. The following properties about $\ell$ and $\leq$ are well-known, e.g. see \cite[Lemma 4.1]{Milicevic21}, \cite[exercise 23, Chapter 2]{Knapp88} and \cite[exercise 21, Chapter 2 ]{BB05} respectively.
\begin{itemize}
    \item Let $\l \in X_*(T)$ be regular dominant, and let $w=ut^\l v \in \wtd W$. Then 
    \begin{equation}\label{len1}
        \ell(w)=\ell(u)+\ell(t^\l)-\ell(v)=\ell(u)+\<2\rho,\l\>-\ell(v).
    \end{equation}
    \item For any element $x$ of $W$, let $\text{Inv}(x)=\{\a\in \Phi^+: x\a \in -\Phi^+\}$. Then for any two elements $x, y \in W$, we have 
    \begin{equation}\label{len2}
        \ell(xy)=\ell(x)+\ell(y)-2|\text{Inv}(x)\cap \text{Inv}(y^{-1})|=\ell(x)-\ell(y)+2|\text{Inv}(x)^c\cap \text{Inv}(y^{-1})|.
    \end{equation}
    \item Let $w_1,w_2,v \in \wtd W$ be three elements such that $\ell(w_iv)=\ell(w_i)+\ell(v)$ for $i=1,2$. Then 
    \begin{equation}\label{len3}
        w_1 \geq w_2 \text{ is equivalent to } w_1v \geq w_2v.
    \end{equation}
\end{itemize}
\subsection{The $\s$-conjugacy classes of $\breve G$}\label{sec:B(G)}
We define the $\s$-conjugation action on $\breve G$ by $g \cdot_\s g'=g g' \s(g) \i$. Let $B(G)$ be the set of $\s$-conjugacy classes on $\breve G$. The classification of the $\s$-conjugacy classes is obtained by Kottwitz in \cite{Kottwitz85} and \cite{Kottwitz97}. Any $\s$-conjugacy class $[b]$ is determined by two invariants: 
\begin{itemize}
	\item Its Kottwitz point $\k([b]) \in \Omega$, and 
	
	\item Its Newton point $\nu([b]) \in X_*(T)_{\mathbb{Q}}$. 
\end{itemize}

We denote by $\leq$ the dominance order on $X_*(T)_{\mathbb{Q}}$ i.e., for $\nu, \nu' \in X_*(T)_{\BQ}$, we say that $\nu \leq \nu'$ if $\nu'-\nu$ is a non-negative (rational) linear combination of positive roots. This upgrades to a partial order on the set $B(G)$. We define $[b_1] \preceq [b_2]$ if $\k([b_1])=\k([b_2])$ and $\nu([b_1])\leq \nu([b_2])$. Let $B(G)_w=\{[b]: [b] \cap \breve I w \breve I \neq \emptyset\}$. The subposet $(B(G)_w,\preceq)$ has an unique maximal element, which coincides with the generic $\s$-conjugacy class in the double coset $\breve I w \breve I$. We denote this maximal element by $[b_w]$ and we set $\nu_w=\nu([b_w])$.

\subsection{Affine Deligne-Lusztig varieties}\label{sec:ADLV}
Let $\mathcal{F}=\breve G/\breve \CI$ be the \emph{affine flag variety}. Recall that for any $b \in \breve G$ and $w \in \wtd W$, the associated \emph{affine Deligne-Lusztig variety} in the affine flag variety is defined as 
$X_w(b)=\{g \breve \CI \in \breve G/\breve \CI; g \i b \s(g) \in \breve \CI w \breve \CI\}$.
We will also discuss certain finite union of affine Deligne-Lusztig varieties. Following \cite{Rapoport05}, we define for a dominant $\mu \in X_*(T)$ its associated \emph{admissible set} as 
\begin{equation*}
    \text{Adm}(\mu)=\{w\in \wtd W: w \leq t^{x (\mu)}\text{ for some } x \in W\}.
\end{equation*}

We will need the following additivity property of admissible sets.
\begin{theorem}\label{adm-add}\cite[Theorem 5.1]{He16}
Let $\mu,\mu' \in X_*(T)$ be dominant. Then we have
\begin{equation*}\label{adm-equal}
    \text{Adm}(\mu) \cdot \text{Adm}(\mu')=\text{Adm}(\mu+\mu').
\end{equation*}
\end{theorem}
For any $b \in \breve G$, we set 
\begin{equation*}
    X(\mu,b)=\bigcup_{w \in \text{Adm}(\mu)} X_w(b).
\end{equation*}

The sets $X_w(b)$ and $X(\mu,b)$ are subschemes, locally of finite type, of the affine flag variety (in the usual
sense in equal characteristic; in the sense of Zhu \cite{Zhu17} in mixed characteristic). Following \cite{KR03}, we define the set of \emph{neutrally acceptable elements for $\mu$} as
\begin{equation*}
    B(G,\mu)=\{[b] \in B(G): \k([b])=\mu, \nu([b])\leq \mu\}.
\end{equation*}

Settling the Kottwitz-Rapoport conjecture made in \cite{KR03} and \cite{Rapoport05} about the non-emptiness pattern for $X(\mu,b)$, He proves the following result in \cite{He16}.
\begin{theorem}\label{KR-conjecture}\cite[Theorem A]{He16}
$X(\mu,b) \neq \emptyset$ if and only if $[b] \in B(G,\mu)$.
\end{theorem}
\subsection{Virtual dimension of affine Deligne-Lusztig variety}\label{sec:virdim}
Note that any element $w \in \wtd W$ can be written in a unique way as $w=u t^\l v$ with $\l$ dominant, $u, v \in W$ such that $t^\l v \in {}^{\BS} \wtd W$. In this case, we set $\eta(w) = vu.$

Let $\mathbf J_b$ be the reductive group over $F$ with $\mathbf J_b(F)=\{g \in \breve G; g b \s(g) \i=b\}.$ Then the \emph{defect} of $b$ is defined by $\rm{def}_{\bG}(b)=\rank_F \bG-\rank_F \mathbf J_b.$ Here for a reductive group $\mathbf H$ defined over $F$, $\rank_F$ is the $F$-rank of the group $\mathbf H$. 

Following \cite[Section 10.2]{He14}, we define the \emph{virtual dimension} to be \[d_{w}(b)=\frac 12 \big( \ell(w) + \ell(\eta(w)) -\rm{def}_{\bG}(b)  \big)-\<\rho, \nu([b])\>.\] 

The justification of defining such an expression lies in a result proved by He in \cite[Theorem 2.30]{He15} that says $\dim X_w(b) \leq d_w(b)$. Recent work by Mili\'cevi\'c and Viehmann in \cite{VM20} singles out those $w \in \wtd W$ for which $\dim X_w(b_w)=d_w(b_w)$ holds; these are called \emph{cordial elements}, and it is shown in loc. sit. that $B(G)_w$ exhibits remarkable properties for such $w$.

\subsection{Demazure product and its variations}\label{sec:Demazure}
We recall three operations  $*, \rhd, \lhd:\widetilde{W} \times \widetilde{W} \rightarrow \widetilde{W}$. Here $*$ is the \emph{Demazure product} and $\lhd, \rhd$ are the left and right \emph{downward Demazure products}, respectively. We describe these operations in form of the following lemma.
\begin{lemma}\label{Dem prod} \cite[Section 2.2]{HeLam15}
Let $x,y \in \widetilde{W}$.
\begin{enumerate}
    \item The subset $\{uv: u\leq x, v\leq y\}$ contains an unique maximal element, which we denote by $x*y$. Moreover, $x*y=u'y=xv'$ for some $u' \leq x$ and $v' \leq y$ and $\ell(x*y)=\ell(u')+\ell(y)=\ell(x)+\ell(v')$.
    \item The subset $\{uy: u \leq x\}$ contains a unique minimal element, which
    we denote by $x \rhd y$. Moreover, $x \rhd y=u''y$ for some $u''\leq x$ with $\ell(x \rhd y)=\ell(y)-\ell(u'')$.
    \item The subset $\{xv: v \leq y\}$ contains a unique minimal element, which
    we denote by $x \lhd y$. Moreover, $x \lhd y=xv''$ for some $v''\leq y$ with $\ell(x \lhd y)=\ell(x)-\ell(v'')$.
\end{enumerate}
\end{lemma}
The Demazure product is related to product of closure of two Iwahori double cosets in $\breve G$. More precisely, for any
$w \in \wtd W , \overline{\breve I w \breve I}= \bigcup_{w' \leq w} \breve I w' \breve I$ is a closed admissible subset of $G$ in the sense of \cite{He16}. Then for any $x,y \in \wtd W$, we have
\begin{equation*}
    \overline{\breve I x \breve I}\cdot \overline{\breve I y \breve I}=\bigcup_{u \leq x}\breve I x \breve I \cdot \bigcup_{v \leq y}\breve I y \breve I=\bigcup_{w \leq x*y}\breve I w \breve I=\overline{\breve I (x*y) \breve I}.
\end{equation*}

The above equation also implies that $*$ is an associative binary operation on $\widetilde{W}$. We will need the following results about these operations.
\begin{lemma}\label{Dem prod prop}\cite[Lemma 2 and Lemma 3]{He09}
\begin{enumerate}
    \item If $w \geq w'$, $v\leq v'$ are elements of $\wtd W$, then $w \rhd v \leq w' \rhd v'$.
    \item For any three elements $x,y,z \in \wtd W$, we have $x \rhd(y \rhd z)=(x * y)\rhd z$. In other words, the action $(\wtd W, *) \times W \rightarrow\wtd W, (x, y) \rightarrow x\rhd y$ is a left action of the monoid $(\wtd W, *)$.
\end{enumerate}
\end{lemma}
\subsection{Quantum Bruhat graph}\label{sec:QBG}
We recall the quantum Bruhat graph introduced by Brenti, Fomin and Postnikov in \cite{BFP99}. By definition, a \emph{quantum Bruhat graph} $\G_{\Phi}$ is a directed graph with 
\begin{itemize}
\item vertices given by the elements of $W$; 

\item upward edges $w\rightharpoonup w s_\a$ for some $\a \in \Phi^+$ with $\ell(w s_\a)=\ell(w)+1$;

\item downward edges $w \rightharpoondown w s_\a$ for some $\a \in \Phi^+$ with $\ell(w s_\a)=\ell(w)-\< 2 \rho, \a^\vee\>+1$. 
\end{itemize}

Note that $\ell(s_\a) \leq \<2\rho,\a^\vee\>-1$ for any $\a \in \Phi^+$, so we have $\ell(w s_\a) \geq \ell(w) - \ell(s_\a) \geq \ell(w) - \< 2\rho,\a^\vee \>+1.$ Therefore the condition for downward edges can be rephrased to saying that $$\ell(w s_\a) = \ell(w) - \ell(s_\a)~\text{with}~\ell(s_\a)=\<2\rho,\a^\vee \>-1.$$ 

Following \cite{LNSSS13}, we call $\a \in \Phi^+$ to be a \emph{quantum root} if $\ell(s_\a)= \<2\rho,\a^\vee\>-1$. We have the following description of quantum roots.

\begin{lemma}\label{qroot} \cite[Lemma 4.2]{LNSSS13}
We have that $\a \in \Phi^+$ is a quantum root if and only if
\begin{enumerate}
    \item $\a$ is a long root, or
    \item $\a$ is a short root, and if $\a=\sum\limits_{\a_i \in \D} c_i \a_i$ then we have $c_i=0$ for any long simple root $\a_i$.
\end{enumerate}
\end{lemma}

Here for simply laced root systems we consider all roots to be long. Thus in a simply laced type, all roots are quantum. Examples of quantum root, in general, include the simple roots as well as the highest root.

\medskip
The weight of an upward edge is defined to be $0$ and the weight of a downward edge $w \rightharpoondown w s_{\a}$ is defined to be $\a^{\vee}$. The \emph{weight of a path} in $\G_{\Phi}$ is defined to be the sum of weights of the edges in the path. For any $x,y\in W$, we denote by $d_\G(x, y)$ the minimal length among all paths in $\G_{\Phi}$ from $x$ to $y$. 

\begin{lemma}\label{wt-x-y} \cite[Theorem 2]{Pos05}, \cite[Lemma 6.7]{BFP99}
Let $x,y\in W$. Then 
\begin{enumerate}
\item There exists a directed path in $\G_{\Phi}$ from $x$ to $y$.

\item All the shortest paths in $\G_{\Phi}$ from $x$ to $y$ have the same weight, which we denote by $\text{wt}(x, y)$. 

\item Any path in $\G_{\Phi}$ from $x$ to $y$ has weight $\ge \wt(x,y)$.
\end{enumerate}
\end{lemma}
\section{Some combinatorial property}
\subsection{Monotonicity of weight function}

Let us define the function $\wt:W \rightarrow X_*(T)$ by $\wt(x):=\wt(x,1)$.
\begin{lemma}\label{monotone-wt}
Let $x_1,x_2 \in W$; if $x_1 \leq x_2$ in Bruhat order, then $ \wt(x_1) \leq \wt(x_2)$ in dominance order.  
\end{lemma}
\begin{proof}
Assume that $\lambda \in X_*(T)^+$ is regular. We first note that $x_1 \leq x_2$ implies $t^\l x_1 \ge t^\l x_2$. It suffices to show this when $x_2$ is a cover of $x_1$, i.e. $x_2=x_1s_\a$ for some $\a \in \Phi^+$ such that $\ell(x_2)=\ell(x_1)+1$. Then $t^\l x_2=t^\l x_1 \cdot s_\a$, hence the two elements in question are comparable. Now, \cref{len1} shows that $$\ell(t^\l x_2)=\ell(t^\l) -\ell(x_2)=\ell(t^\l)-\ell(x_1)-1=\ell(t^\l x_1)-1.$$
Hence $t^\l x_1 \ge t^\l x_2$.

Recall the description of maximal Newton point associated to $w \in \wtd W$ given in \cite{ Vieh14}.
\begin{equation}\label{max-nu}
    \nu_w=\max\{\nu(u):u\leq w, u \in \widetilde{W}\}.
\end{equation}

By applying this, we get $\nu_{t^\l x_1} \ge \nu_{t^\l x_2}$. Let us now assume that $\l$ is actually superregular; hence the formula for maximal Newton point in \cite[Theorem 3.2]{Milicevic21} applies in our case and it gives $\nu_{t^\l x_i}= \l - \wt(x_i^{-1})= \l - \wt(x_i)$ for $i=1,2$. Thus we get $\l - \wt(x_1)\geq \l - \wt(x_2)$, whence $\wt(x_1) \le \wt(x_2)$. 
\end{proof}

\subsection{Reduction to dominant chamber alcoves}\label{sec:reduction}
We will now reduce our calculation of maximal Newton point of an arbitrary element of $\widetilde{W}$ to that of a suitable element lying in the dominant chamber. 

\begin{proposition}\label{reduction}
If $\lambda$ is dominant regular, then $[b_{ut^{\lambda}v}]=[b_{t^{\lambda}(v \lhd u)}]$.
\end{proposition}
\begin{proof}
Note that by \cref{len1}, $\ell(ut^\l v)=\ell(u)+\ell(t^\l v)$, hence $ut^\l v = u*(t^\l v)$. Thus $\overline{\breve I ut^\l v \breve I}=\overline{\breve I u \breve I}\cdot \overline{\breve I t^\l v \breve I}$.

Now recall that $\breve I$ is $\sigma$-stable, and since $\bG$ is split, $\sigma$ acts trivially on $\wtd W$ too. Hence $\s$ fixes $\overline{\breve I w \breve I}$ for any $w \in \wtd W$. Therefore, if $w=w_1w_2$ is an element with $w_1\in \overline{\breve I u \breve I}, w_2 \in \overline{\breve I t^\l v \breve I}$, then $w_1^{-1} \cdot_\s w=w_2w_1$. This shows that $\breve G \cdot_{\sigma} (\overline{\breve I u \breve I} \cdot \overline{\breve I t^\l v \breve I}) \subset \breve G\cdot_{\sigma} (\overline{\breve I t^\l v \breve I} \cdot \overline{\breve I u \breve I})$. One obtains the other inclusion in a similar way. This shows that $\breve G \cdot_{\sigma} (\overline{\breve I u \breve I} \cdot \overline{\breve I t^\l v \breve I}) = \breve G\cdot_{\sigma} (\overline{\breve I t^\l v \breve I} \cdot \overline{\breve I u \breve I}).$ 

Hence, we conclude that $\breve G \cdot_\s \overline{\breve I ut^\l v \breve I}=\breve G \cdot_\s \overline{\breve I (t^\l v)*u \breve I}.$
Since $G\cdot_{\sigma}\overline{\breve I w \breve I}$ is the set of $\sigma$-conjugacy classes intersecting $\overline{\breve I w \breve I}$, this shows that \begin{equation}\label{red}
    [b_{ut^\l v}]=[b_{(t^\l v)*u}].
\end{equation}

Thus it suffices to show that

(a) if $\l$ is dominant regular, then $(t^\l v)*u = t^\l (v \lhd u)$.

By \cref{Dem prod prop}, we just need to argue the case when $u$ is a simple reflection $s\in \BS$. Then this statement is equivalent to the assertion that $\ell(vs)=\ell(v)-1$ if and only if $\ell(t^\l vs)=\ell(t^\l v)+1$. This follows directly from a computation using \cref{len1}. Hence we are done.
\end{proof}
\begin{corollary}\label{wt}
Let $x,y \in W$. Then $\wt(x,y)=\wt(x^{-1}\lhd y)$.
\end{corollary}
\begin{proof}
Suppose that $\l \in X_*(T)^+$ is superregular. Applying the formula established in \cite[Theorem 3.2]{Milicevic21} to $w:=x t^\l y$ and $w':=t^\l (y \lhd x)$, we get that $\nu_w=\l - \wt(y^{-1},x)$ and $\nu_{w'}=\l - \wt((y \lhd x)^{-1})=\l-\wt(y \lhd x)$. Since $\nu_w=\nu_{w'}$ by \cref{reduction}, we get $\wt(y^{-1},x)=\wt(y \lhd x)$. Now replacing the pair $(x,y)$ by $(y,x^{-1})$, we get the conclusion. 
\end{proof}

\subsection{Reformulation of weight in quantum Bruhat graph}
We now give another interpretation of weight of a shortest path in the quantum Bruhat graph.
\begin{definition}\label{RQRD}
For an element $x \in W$, we say that $x=s_{\beta_1}\cdots s_{\beta_k}$ is a \textit{reduced quantum reflection decomposition} if
\begin{enumerate}
    \item $\{\b_i:1 \leq i \leq k\}$ is a collection of (not necessarily simple) quantum roots, i.e. $l(s_{\b_i}) = \< 2\rho,\b_i^\vee\> -1.$
    \item $\ell(x)=\sum_{i=1}^k \ell(s_{\b_i})$.
    \item Subject to the first two conditions, $k$ is minimal.
\end{enumerate}
In this case, we say that $k$ is the minimal reduced length of $x$ in terms of reflections associated to quantum roots, and write $\ell_\downarrow(x)=k$. Note that $d_{\G}(x,1)=\ell_\downarrow(x)$.
\end{definition}

The choice of notation is suggested from the fact that 
\begin{center}
    \{Reduced quantum reflection decompositions for $x$\} 
    
    \big\updownarrow 
    
    \{Shortest paths in $\G_\Phi$ from $x$ to $1$ that uses \textit{only downwards edges}\}
\end{center}
By \cite[Proposition 4.11]{VM20}, the weight of any path in the later set is equal to $\wt(x)$. Therefore, the problem of computing $\wt(x)$ reduces to one of finding a suitable decomposition for $x$: if $x=s_{\b_1}\cdots s_{\b_k}$ is such a decomposition then $\wt(x)=\sum_{i=1}^k \b_i^\vee$. In view of \cref{wt}, we can thus reformulate $\wt(x, y)$ in terms of a decomposition of $x^{-1} \lhd y$.

\section{Formula for the maximal Newton point}\label{sec:gen-nu}
The goal of this section is to prove the following result. 
\begin{theorem}\label{Generic-nu}
Assume that $\bG$ is a quasi-simple split group of rank $n$. Let $w=t^\l x$ be an element of its Iwahori-Weyl group such that $\l$ is dominant and $\text{depth}(\l) > \Xi $, where $\Xi$ is an integer depending on $n$. Then the maximal Newton point associated to $w$ is given by $\nu_w=\l-\wt(x)$.
\end{theorem}
For each Cartan type, the lower bound $\Xi$ for the depth mentioned above is given in the following table.

\begin{center}\label{final bd}
\begin{tabular}{ |c|c|c|c|c|c|c|c|c|} 
 \hline
 Type  & $A_n$ & $B_n/C_n$ & $D_n$ & $E_6$ & $E_7$ & $E_8$ & $F_4$ & $G_2$  \\
 \hline
 $\Xi$ & $3n+1$ & $6n-2$ & $6n-6$ & $23$ & $33$ & $57$ & $23$ & $9$  \\
 \hline
\end{tabular}
\end{center}
\begin{remark}\label{reduction1}
We note that one can handle the case of a split connected reductive group $\bG$ via a routine reduction procedure, which we discuss now. In that case, we have $$W=W_1 \times \cdots \times W_l,$$ where $W_i$ are irreducible Weyl groups. This gives rise to a partition of $m$, where $m$ is the semisimple rank of $\bG$. Any element $x$ of $W$ is of the form $(x_1,\cdots, x_l)$. Similarly, we write $\l$ as $(\l_1,\cdots,\l_l)$. We have $w_0=(w_{0,1},\cdots,w_{0,l})$, where $w_{0,i}$ is the longest element of $W_i$. In the same way, we write the sum of positive roots as $(2\rho_1^\vee,\cdots, 2\rho_l^\vee)$ and the highest root as $(\theta_1,\cdots,\theta_l)$.

We observe that while dealing with the quasi-simple case, we introduce a restriction on depth in two stages - as found in \cref{sec:bd1} and \cref{sec:keylemma}. By the last paragraph we therefore need to require for each $i$ $$\text{depth}(\l_i) \geq \Xi_i,$$ where $\Xi_i$ is the lower bound given for the Weyl group $W_i,$ to be read off from the table above. Note that $\Xi_i$ is then some linear expression of $m$, depending on the type of $i$-th factor and the partition of $m$. We remark that under such hypothesis, \cref{Generic-nu} applies to each quasi simple factor of $\bG$ and thus we need to justify the formula for such groups.

By \cref{reduction}, we know that $\nu_{ut^\l v}=\nu_{t^\l (v \lhd u)}$. Then by \cref{Generic-nu} it follows that $\nu_{t^\l (v \lhd u)}=\l - \text{wt}(v \lhd u)$. By \cref{wt} we then conclude that $\nu_w=\l-\text{wt}(v^{-1},u)$. This concludes our discussion.
\end{remark} 
\subsection{Proposed linear bound on depth}\label{sec:bd1}
For the rest of \cref{sec:gen-nu}, we assume that $W$ is an irreducible Weyl group for the root system of Cartan type $X_n$. We first establish an upper bound for 
\begin{equation*}
    \CM=\CM_{X_n}:=\max\{\<\alpha_i, \wt(x) \>: \alpha_i \in \D, x \in W\}
\end{equation*}
This is based on the explicit computation of $\text{wt}(w_0)$ for the longest element $w_0$ in each type, which we defer to the next section for the sake of continuity.

\smallskip
By \cref{monotone-wt}, we have $\{\wt (x): x \in W\} \subset \{\nu \in X_*(T): \nu \leq \wt(w_0)\}$. Therefore, $$\CM \leq \max \{\< \a_i,\nu\>: \nu \in X_*(T), \nu \leq \wt(w_0), \a_i \in \D\}.$$ It is easy to see that maximum of the latter set is realized at $\nu=\kappa_l \a_l^\vee$, where $\kappa_l \a_l^\vee$ is a summand of $\wt(w_0)$ such that $\kappa_l=\max\{\kappa_j:\kappa_j\a_j^\vee~\text{is a summand of $\wt(w_0)$}\}$. Since the maximum value of $\<\a_i,\a_l^\vee\>$ is $2$ in every Cartan type, this shows that $\CM_{X_n}$ is bounded above by twice the largest coefficient in the expression of $\wt(w_0)$ in each type $X_n$. This is tabulated in the list below, where we denote by $\wtd \CM=\wtd \CM_{X_n}$ the upper bound obtained in this manner.

\smallskip
\begin{center}\label{1st bd}
\begin{tabular}{ |c|c|c|c|c|c|c|c|c|} 
 \hline
 Type  & $A_n$ & $B_n/C_n$ & $D_n$ & $E_6$ & $E_7$ & $E_8$ & $F_4$ & $G_2$  \\
 \hline
 $\wtd \CM$ & $n+1$ & $2n$ & $2n$ & $12$ & $16$ & $28$ & $12$ & $4$  \\
 \hline
\end{tabular}
\end{center}
\subsection{Proof of the easier inequality}\label{sec:easy>}
We need the following result about Bruhat order on $\wtd W$.
\begin{lemma}\cite[Remark 3.9]{Rapoport05}\label{Base case}
Let $\l \in X_*(T)$ be dominant. Let $\b \in \Phi^+$ such that $\l-\b^\vee$ is dominant. Then $t^{\l-\b^\vee} \leq t^\l s_\b \leq t^\l$ in $\wtd W$.
\end{lemma}
We start by establishing the easier inequality.
\begin{proposition}\label{easy >}
Let $w=t^\l x\in \wtd W$ such that $\text{depth}(\l) \geq \wtd \CM$. Then $w \geq t^{\l - \wt(x)}$ and therefore $\nu_w \geq \l-wt(x)$.
\end{proposition}
\begin{proof}
Suppose that $x=s_{\b_1} \cdots s_{\b_l}$ is a reduced quantum reflection decomposition with $\ell_\downarrow(x)=l>1$, and therefore $\wt (x)=\sum\limits_{i=1}^{l} \b_i^\vee$. Let us first note that:
\[\text{(a) $s_{\b_1}\cdots s_{\b_k}$ is a reduced quantum reflection decomposition for $x_k:=xs_{\b_l}\cdots s_{\b_{k+1}}$ for all $k \leq l$.}\] Suppose otherwise; then we can find a reduced quantum reflection decomposition of the form $x_k=s_{\g_1}\cdots s_{\g_j}$ with $j < k$. But then $x=x_ks_{\b_{k+1}}\cdots s_{\b_l}=s_{\g_1}\cdots s_{\g_j}s_{\b_{k+1}}\cdots s_{\b_l}$ satisfies length additivity: $$\ell(x)=\ell(x_k)+\sum\limits_{i=k+1}^l \b_i^\vee=\sum\limits_{i=1}^{j} \g_i^\vee+\sum\limits_{i=k+1}^{l} \b_i^\vee.$$ Note that all associated roots in this new decomposition of $x$ are quantum, and it has $j+l-k <l$ factors, hence it contradicts the fact that $\ell_\downarrow(x)=l$. So statement (a) is proved.

To prove the proposition, we induct on $\ell_\downarrow(x)$. We note our depth hypothesis ensures that $\l-\b^\vee$ is dominant for all $\b \in \Phi^+$: for $\a_i \in \D$, we have $\<\a_i,\l-\b^\vee\>=\<\a_i,\l\>-\<\a_i,\b^\vee\> > 0$, since maximum value of $\langle \alpha^{\vee},\alpha_i \rangle$ equals $2$ in every Cartan type, except for $G_2$ - in which case it equals $3$, cf. \cite[Chapter 6, Section 1, no. 3]{Bou}. Therefore \cref{Base case} covers the base case, i.e. when $x=s_\b$ for some $\b \in \Phi^+$.

Let us now apply induction hypothesis to $xs_{\b_l}=s_{\b_1} \cdots s_{\b_{l-1}}$. Since $\wt(xs_{\b_l})=\sum\limits_{i=1}^{l-1} \b_i^\vee$ by virtue of statement (a), this gives 
\begin{equation}\label{hypo}
    t^\l s_{\b_1} \cdots s_{\b_{l-1}} \geq t^{\l-\sum\limits_{i=1}^{l-1} \b_i^\vee}.
\end{equation}
To perform the induction step, we apply a property of Bruhat order as described in \cref{len3}. We set $$w_1=t^{\l}x, w_2=t^{\l-\sum\limits_{i=1}^{l-1} \b_i^\vee}s_{\b_l}, v=s_{\b_l}$$ and check the required length additivity. Note that $\l-\sum\limits_{i=1}^{l-1} \b_i^\vee$ is dominant under our depth hypothesis: for $\a_i \in \D$, \[\<\a_i,\l-\sum\limits_{i=1}^{l-1} \b_i^\vee\>=\<\a_i,\l\>-\<\a_i,\wt(xs_{\b_l})\> > \<\a_i,\l\>-\CM \geq 0,\]therefore the length formula in \cref{len1} applies. We see that
\begin{enumerate}
    \item $\ell(w_1v)=\ell(t^\l s_{\b_1} \cdots s_{\b_{l-1}})=\ell(t^\l)-\sum\limits_{i=1}^{l-1}\ell(s_{\b_i})=\ell(t^\l)-\sum\limits_{i=1}^l\ell(s_{\b_i})+\ell(s_{\b_l})=\ell(w_1)+\ell(v).$
    \item $\ell(w_2v)=\ell(t^{\l-\sum\limits_{i=1}^{l-1} \b_i^\vee})=\ell(t^{\l-\sum\limits_{i=1}^{l-1} \b_i^\vee})-\ell(s_{\b_l}) + \ell(s_{\b_l})=\ell(w_2)+\ell(v).$
\end{enumerate}
In presence of \cref{hypo}, we therefore get $w_1 \geq w_2$, i.e. $t^{\l}x \geq t^{\l-\sum\limits_{i=1}^{l-1} \b_i^\vee}s_{\b_l}$. Finally, we note that $\l-\wt(x)=\l-\sum\limits_{i=1}^{l} \b_i^\vee$ is dominant (by similar argument as before) and thus \cref{Base case} applies to give $t^{\l-\sum\limits_{i=1}^{l-1} \b_i^\vee}s_{\b_l} \geq t^{\l-\sum\limits_{i=1}^{l} \b_i^\vee}$. Combining these inequalities, we get $t^\l x \geq t^{\l -\wt(x)}$. Appealing to \cref{max-nu}, we thus get $\nu_w \geq \l-\wt(x)$. This finishes the proof.
\end{proof}
\begin{lemma}\label{translations}
Let $w=t^\l x\in \wtd W$ such that $\text{depth}(\l) > \wtd \CM$. Then $\nu_w=\max\{\g^+:t^\g \leq w\}$.
\end{lemma}
\begin{proof}
By \cref{max-nu}, we can assume that $\nu_w=\nu(t^{y\mu}z)$ for some $t^{y\mu}z \in \wtd W, w \geq t^{y\mu}z$. Suppose that $z \neq 1$, therefore the order of $z$ equals $m \geq 2$. Then $\nu(t^{y\mu}z)=(\frac{1}{m}\sum\limits_{i=1}^m z^i(y\mu))^+$. Note that $\z \cdot \sum\limits_{i=1}^m z^i(y\mu) = \sum\limits_{i=1}^m z^i(y\mu)$ for all $\z \in W$, so the element $\sum\limits_{i=1}^m z^i(y\mu)$ lies on the wall of some Weyl chamber, cf. \cite[Chapter 5, Section 3.3, Remark 3]{Bou}. Therefore $(\frac{1}{m}\sum\limits_{i=1}^m z^i(y\mu))^+$ lies on the wall of $\CC^+$, and hence $\nu_w$ is singular. 

By \cref{easy >}, we can assume that $\nu_w=\l-\e$, with $\e \leq \wt(x) \leq \wt(w_0)$. Hence, there is some $\a_i \in \D$ such that $\<\a_i,\l-\e\>=0$, and thus $\<a_i,\l\>=\<\a_i,\e\> \leq \wtd M$, which is a contradiction to hypothesis on $\text{depth}(\l)$. Therefore, $\nu_w$ cannot be singular and hence $z=1$. 
\end{proof}
\subsection{Toward computing a downward Demazure product}
We need to understand the largest translation dominated by $w$ to determine its associated generic Newton point $\nu_w$. In view of \cref{easy >}, let us denote this translation element by $t^{y\g}$ for some $y \in W$ and $\g \geq \l-\wt(x)$. Of course, this would be equality if $\l$ is superregular. We want to leverage this result to establish such equality with some linear bound on depth here. The following lemma lay the path towards proving that. 
\begin{lemma}\label{lifting}
If $\l-2\rho^\vee$ is dominant, we have $t^\l x\geq t^{y\g}$ if and only if $t^{\l-2\rho^\vee}\geq t^{-2\rho^\vee}\rhd t^{y\g}$.
\end{lemma}
\begin{proof}
Note that if $\l-2\rho^\vee$ is dominant, $\ell(t^{\l-2\rho^\vee})=\<2\rho,\l-2\rho^\vee\>=\ell(t^\l)-\ell(t^{-2\rho^\vee})$; also, since $\l\geq \l-2\rho^\vee$ are dominant coweights, $t^\l \geq t^{\l-2\rho^\vee}$, cf. \cite[proof of Proposition 3.5]{Rapoport05}. Hence, $t^{-2\rho^\vee}\rhd t^\l=t^{\l-2\rho^\vee}$. 

Now, it suffices to show the following:

(a) If $a,b$ are two elements in a Coxeter group and $s$ is a simple reflection, then $a \geq b$ if and only if $s \rhd a \geq s \rhd b$.

This follows from the well known lifting property of Bruhat order.
Note that we only need to consider the case when $s \rhd a = sa$, i.e. $a \geq sa$. Then there are two further cases.
\begin{enumerate}
    \item $s \rhd b = b$: in this case, $sb \geq b$. Hence $a\geq b$ is equivalent to $sa \geq b$, and hence $s \rhd a \geq s \rhd b$.
    \item $s \rhd b = sb$: in this case, $b \geq sb$. Hence $a \geq b$ is equivalent to $sa \geq sb$, and hence $s \rhd a \geq s \rhd b$.
 
\end{enumerate}
\end{proof}

In \cref{sec:keylemma} we will focus on computing the downward Demazure product that appears in \cref{lifting}. Before proceeding any further to do that, let us first sketch what the answer should look like. We first apply \cref{Dem prod prop} with $w=t^{-2\rho^\vee},w'=w_0$ and $v=v'=t^{y\mu}$; if we require that $\l$ is of large enough depth (in a sense made precise below) so that $\g$ is dominant regular, we get $w_0 \rhd yt^\g y^{-1} =t^\g y^{-1} \geq t^{-2\rho^\vee}\rhd t^{y\g}$ . We next apply \cref{Dem prod prop} with $v'=t^{y\g},v=t^{\g}y^{-1}$ and $w=w'=t^{-2\rho^\vee}$. If we now require $\l$ to be of sufficiently large depth so that $\g-2\rho^\vee$ is also dominant, we get $t^{-2\rho^\vee}\rhd t^{y\g} \geq t^{-2\rho^\vee}\rhd t^\g y^{-1}=t^{\g-2\rho^\vee}y^{-1}$. Combining these, we see that whenever $\l$ has ``large" depth (to be specified later) we have $$t^\g y^{-1} \geq t^{-2\rho^\vee}\rhd t^{y\g} \geq t^{\g-2\rho^\vee}y^{-1}.$$

Therefore, it is natural to expect that this downward Demazure product depends only on $y$ whenever $\l$ has a ``large" depth. In the next section, we quantify this depth condition and show that we indeed get the desired conclusion. 

\subsection{A technical lemma}\label{sec:keylemma}
For each irreducible Cartan type $X_n$, we let $\CS=\CS_{X_n}$ be twice the sum of all coefficients appearing in the expression of the highest root in terms of simple roots; in other words, $\CS=\<\theta,2\rho^\vee\>$, where $\theta$ is the highest root. This quantity will become relevant for our next lemma. We record this integer for each type in the table below.\\
\begin{center}
\begin{tabular}{ |c|c|c|c|c|c|c|c|c|} 
 \hline
 Type & $A_n$ & $B_n/C_n$ & $D_n$ & $E_6$ & $E_7$ & $E_8$ & $F_4$ & $G_2$  \\
 \hline
 $\CS$ & $2n$ & $4n-2$ & $4n-6$ & $11$ & $17$ & $29$ & $11$ & $5$  \\
 \hline
\end{tabular}
\end{center}

\smallskip
Note that $\Xi=\wtd \CM+\CS$ is the final lower bound that appears in \cref{Generic-nu}. We now prove a key lemma that would help us achieve the harder inequality.
\begin{lemma}\label{independence}
Let $\text{depth}(\l) > \Xi$. We continue to assume that $w=t^\l x\geq t^{y\g}$ such that $\g=\nu_w \geq \l-\wt(x)$. Then there exists a coweight $\mu_y$ depending only on $y$ such that $t^{-2\rho^\vee}\rhd t^{\g y}=t^{\g-\mu_y}y^{-1}$.
\end{lemma}
\begin{proof}
We begin by noting that we can compute this downward Demazure product from a specific length additive decomposition of $t^{y\g}$. More precisely, we have the following desiderata
\begin{itemize}
    \item (D1) a decomposition $t^{y\g}=a_1t^{\mu_1}b_1 \cdot a_2t^{\mu_2}b_2$, where $\mu_i$ is dominant and $t^{\mu_i}b_i \in~ ^\BS\wtd W$ for $i=1,2$, and 
    \item (D2) $\ell(t^{y\g})=\ell(a_1t^{\mu_1}b_1) + \ell(a_2t^{\mu_2}b_2)$, and
    \item (D3) $a_1t^{\mu_1}b_1$ is the largest element dominated by $t^{2\rho^\vee}$, subject to the first two conditions.
\end{itemize}
\Cref{Dem prod} asserts the existence of such decomposition, and then we have $t^{-2\rho^\vee}\rhd t^{\g y}=a_2t^{\mu_2}b_2$. We now proceed to identify these elements $a_it^{\mu_i}b_i$ in the following three steps.

\subsubsection{Relating relevant coweights} We first determine how the coweights associated to translation part of these elements are related. By (D1), $t^{y\g}=a_1b_1a_2t^{a_2^{-1}b_1^{-1}\mu_1+\mu_2}b_2$, so we have $\g=(a_2^{-1}b_1^{-1}\mu_1+\mu_2)^+$, i.e. there exists some $v \in W$ such that $\g=v(a_2^{-1}b_1^{-1}\mu_1+\mu_2)$. We will now show that $v=1$.

Assume the contrary and let $\beta \in \text{Inv} (v)$. Now, $\mu_2=v^{-1}\g-a_2^{-1}b_1^{-1}\mu_1$ is dominant, therefore $$\<\b, v^{-1}\g-a_2^{-1}b_1^{-1}\mu_1\> \geq 0.$$ 

By (D3), we have $t^{2\rho^\vee}\geq a_1t^{\mu_1}b_1$, thus $2\rho^\vee \geq \mu_1$; by a geometric characterization of dominance order, e.g. see \cite[Lemma 12.14]{AB83}, this gives \[\mu_1=\sum\limits_{\z \in W}a_\z \z\cdot 2\rho^\vee~\text{for some}~ a_\z \in \BR_{\geq 0}~\text{such that}~\sum\limits_{\z \in W} a_\z=1.\] 

Substituting this expression of $\mu_1$ in the pairing above and using its $W$-invariance, we get
\begin{equation*}
    \<v\b,\g\>-\sum_{\z \in W} a_\z \< \z^{-1} b_1a_2\b, 2\rho^\vee \> \geq 0.
\end{equation*}

Let $\b'=-v\b \in \Phi^+$. Note that $\z^{-1}b_1a_2\b \geq -\theta$ for any element $\z \in W$, and therefore $\<\z^{-1}b_1a_2\b, 2\rho^\vee\> \geq -\<\theta,2\rho^\vee\>$. Putting all these together, we obtain 
\begin{equation*}
    \<\b',\g\>=-\sum_{\z \in W} a_\z \< \z^{-1} b_1a_2\b, 2\rho^\vee \>\leq \sum\limits_{\z \in W}a_\z \<\theta,2\rho^\vee\>=\CS.
\end{equation*}

Now recall that $\g=\nu_w=\l-\e$ is dominant with $\e \leq \wt(x) \leq \wt(w_0)$. Choose a simple root $\a \leq \b'$. Then we get $\<\a,\g\>\leq \<\b',\g\>\leq \CS$, and therefore $\<\a,\l\>\leq \CS+\<\a,\e\>\leq \CS+\wtd \CM$. Since this yields a contradiction to the depth hypothesis of $\l$, we are done. Therefore $v=1$ and $\g=a_2^{-1}b_1^{-1}\mu_1+\mu_2$.

\subsubsection{Showing that certain coweights are regular} Note that $\g$ is regular, since for any simple  root $\a_i$ we have $\<\a_i,\g\>=0 \implies \<\a_i,\l\>=\<\a_i,\e\> \leq \wtd \CM$, thereby contradicting the depth hypothesis imposed on $\l$. Hence, knowing that the translation parts in $t^{y\g}=a_1b_1a_2t^{a_2^{-1}b_1^{-1}\mu_1+\mu_2}b_2$ are equal allows us to relate the finite Weyl group components from both sides. Namely, we have $y=a_1b_1a_2=b_2^{-1}$. We can simplify these relations further in the following way. We observe that $\mu_2$ is regular, since otherwise we can find a simple root $\a$ such that $\<\a,\g\>=\sum_{\z \in W}a_\z\<\a, \z\cdot a_2^{-1}b_1^{-1}2\rho^\vee\>=\sum_{\z \in W} \<b_1a_2\z\a,2\rho^\vee\>\leq \sum_{\z \in W}a_\z \< \theta,2\rho^\vee\> = \CS$; but this would yield contradiction to the imposed depth condition as before. 

\subsubsection{Showing that the product describes a dominant chamber alcove} We can now show that $a_2=1$. Recall that $t^{-2\rho^\vee}\rhd t^{\g y}=a_2t^{\mu_2}b_2$. Let us rewrite $t^{-2\rho^\vee}=w_0t^{2\rho^\vee}w_0=w_0t^{2\rho^\vee}w_0$ and apply $w_0 \rhd$ to this equation. This gives
\[ w_0 \rhd \{(w_0t^{2\rho^\vee}w_0) \rhd t^{y\g}\}=w_0\rhd (a_2t^{\mu_2}b_2).\]

We now apply \cref{Dem prod prop} on the left hand side, and use the fact $\mu_2$ is regular in the right hand side. We get
\[\{w_0*(w_0t^{2\rho^\vee}w_0)\}\rhd t^{y\g}=(w_0\rhd a_2)t^{\mu_2}b_2.\]

Note that $w_0t^{2\rho^\vee}w_0=w_0*t^{2\rho^\vee}w_0$, and hence we can use associativity of operation $*$ to rewrite the element in the parenthesis on the left hand side as $w_0*(w_0*t^{2\rho^\vee}w_0)=(w_0*w_0)*t^{2\rho^\vee}w_0=w_0*t^{2\rho^\vee}w_0=w_0t^{2\rho^\vee}w_0=t^{-2\rho^\vee}$. Therefore, we finally deduce 
\[t^{-2\rho^\vee}\rhd t^{y\g}=t^{\mu_2}b_2.\]

This proves our claim.

\medskip
Therefore, the relations arising from (D1) are
\begin{equation}\label{D1}
    \g=b_1^{-1}\mu_1+\mu_2, y=a_1b_1=b_2^{-1}.
\end{equation}

Let us now utilize (D2). It says that we have length additivity in the decomposition 
$t^{y\g}=a_1t^{\mu_1}b_1 \cdot t^{\mu_2}b_2$, i.e.
\begin{equation*}
    \<\g,2\rho\>=\ell(a_1)+\<\mu_1,2\rho\>-\ell(b_1)+\<\mu_2,2\rho\>-\ell(b_2).
\end{equation*}

By use of \cref{D1}, the last equation reduces to 
\begin{equation}
    \ell(a_1b_1)=\ell(a_1)-\ell(b_1)+\<\mu_1-b_1^{-1}\mu_1,2\rho\>.
\end{equation}

Therefore we are led to consider the following subset of $\wtd W$ $$\Omega_y=\{at^{\mu}b \in \wtd W:y=ab, at^{\mu}b\leq t^{2\rho^\vee}, \ell(ab)=\ell(a)-\ell(b)+\<\mu-b^{-1}\mu,2\rho\>\}.$$

Note that the definition of $\Omega_y$ encapsulates all information
amongst $a_1,b_1,\mu_1$ coming out of the first two items in our desiderata. By (D3), the element $a_1t^{\mu_1}b_1$ is therefore uniquely specified as the maximal element of $\Omega_y$. Hence, the element $b_1^{-1}\mu_1$ depends only on $y$. Denoting it by $\mu_y$, we see that \cref{D1} implies $\mu_2 = \g - \mu_y$. This finishes the proof.
\end{proof}
\subsection{Finishing the proof}
\begin{proof}[Proof of \Cref{Generic-nu}] By combining \cref{lifting} and \cref{independence}, we observe that our standing assumption $ t^{\l}x \geq t^{y\g}$ implies
\begin{equation*}
    t^{\l-2\rho^{\vee}}x \geq t^{\g - \mu_y}y^{-1}.
\end{equation*}

Let us now choose $\l'$ to be large enough so that $\l+\l'$ is superregular, e.g. take $\l'=n\rho^\vee$ for large $n$. We note that 
\begin{enumerate}
    \item $\ell(t^{\l'}\cdot t^{\l-2\rho^{\vee}}x)=\ell(t^{\l+\l'-2\rho^{\vee}}x)=\<2\rho,\l+\l'-2\rho^{\vee}\>-\ell(x)=\<2\rho,\l'\>+\<2\rho,\l-2\rho^{\vee}\>-\ell(x)=\ell(t^{\l'})+\ell(t^{\l-2\rho^{\vee}}x).$ 
    \item $\ell(t^{\l'} \cdot t^{\g - \mu_y}y^{-1})=\ell(t^{\l'+\g - \mu_y}y^{-1})=\<2\rho,\l'+\g - \mu_y\>-\ell(y^{-1})=\<2\rho,\l'\>+\<2\rho,\g-\mu_y\>-\ell(y^{-1})=\ell(t^{\l'})+\ell(t^{\g - \mu_y}y^{-1}).$ 
\end{enumerate}

Hence, we can apply \cref{len3} in this situation, and get 
\begin{equation*}
     t^{\l+\l'-2\rho^{\vee}}x \geq t^{\g+\l' - \mu_y}y^{-1}.
\end{equation*}

Now, we use the `only if' part of \cref{lifting} to deduce from the previous equation
\begin{equation*}
    t^{\lambda+\lambda'}x \geq t^{y(\g+\l')}.
\end{equation*}

Therefore, by \cref{max-nu} we have $\nu_{t^{\l+\l'}x} \geq \nu_{t^{y(\g+\lambda')}}$. Since the formula for the maximal Newton point in \cite[Theorem 3.2]{Milicevic21} applies to elements of $\wtd W$ having $\l+\l'$ as its dominant translation part, we get that $\l+\l'-\wt(x) \geq \g+\lambda'$, whence $\l-\wt(x) \geq \g=\nu_w$. Combining this with \cref{easy >}, we get $\nu_w=\l-\wt(x)$. This finishes the proof.
\end{proof}
\subsection{A remark about Bruhat order on $\wtd W$}
\begin{proposition}\label{alt}
Suppose that $\bG$ is a quasi-simple split group. Assume $\text{depth}(\l) > \Xi$. Then $ut^\l v \geq t^{y\g}$ for some dominant $\g$ implies that $t^\l (v \lhd u) \geq t^\g$.
\begin{proof}
Note that $ut^\l v \geq t^{y\g}$ implies $\nu_{ut^\l v} \geq \nu_{t^{y\g}}$ by \cref{max-nu}. By \cref{Generic-nu} and \cref{wt}, we then have $\l-\text{wt}(v \lhd u) \geq \g$. By \cref{easy >}, we have that $t^\l (v \lhd u) \geq t^{\l-\text{wt}(v \lhd u)}$. Since $\l-\text{wt}(v \lhd u), \g$ are both dominant, we also have that $t^{\l-\text{wt}(v \lhd u)} \geq t^\g$, cf. \cite[proof of Proposition 3.5]{Rapoport05}. Putting together the last two inequalities, we finally get $t^\l (v \lhd u) \geq t^\g$.
\end{proof}
\begin{remark}\label{speculation}
Note that if $u=1$, this says that $t^\l v \geq t^{y\g}$ implies $t^\l v \geq t^\g$. In other words, if some $W$-conjugate of a dominant translation element lies below an element in the dominant chamber, so does the dominant translation element itself. We point out that in the presence of \cref{easy >} and the fact that maximal Newton point formula is available for elements with superregular translation part, this is equivalent to \cref{Generic-nu}. Hence, an independent proof of \cref{alt} would help us get rid of the bound imposed in \cref{sec:keylemma}; similarly, an improved estimate for $\CM$ would weaken the final depth hypothesis required in \cref{Generic-nu}.
\end{remark}
\end{proposition}

\section{Weight of the longest element}
In this section, we list out reduced quantum reflection decomposition for the longest element $w_0$ in Weyl group $W$ of each irreducible Cartan type. For an element $x \in W$, the \textit{reflection length} of $x$ is the smallest number $l$ such that $x$ can
be written as a product of $l$ reflections in $W$. We denote by $\ell_R(x)$ the reflection
length of $x$. Note that in general $\ell_R(x) \leq \ell_\downarrow(x)$, and strict inequality can occur.

We enumerate reflection length for $w_0$ for all the simple groups below.
\begin{center}\label{ref length}
\begin{tabular}{ |c|c|c|c|c|c|c|c|c|} 
 \hline
 Type & $A_n$ & $B_n/C_n$ & $D_n$ & $E_6$ & $E_7$ & $E_8$ & $F_4$ & $G_2$  \\ 
 \hline
 $\ell_R(w_0)$ & $\lceil \frac{n}{2} \rceil$ & $n$ & $2 \lfloor \frac{n}{2} \rfloor$ & $4$ & $7$ & $8$ & $4$ & $2$  \\
 \hline
\end{tabular}
\end{center}

\smallskip
We compute $\wt(w_0)$ by exhibiting suitable decomposition of $w_0$ in each type, and on the way we see that $\ell_R(w_0)=\ell_\downarrow(w_0)$. We first write down expression of the longest element affording its reflection length; for the classical types, we extract this from \cite{Blanco}, and for the exceptional ones we can check this directly by hand or using a computer algebra system such as \cite{Sagemath}. Once we find these decompositions, we see that they only involve reflections corresponding to quantum roots; furthermore, these expressions do satisfy the length additivity condition as well. Taken together, these observations establish that we have found reduced quantum reflection decomposition for $w_0$ in these types. Finally, we add up the coroots associated to the reflections appearing in each such decompositions and deduce $\wt(w_0)$ from that. See \cref{cascade} for an alternative recipe for computing $\wt(w_0)$.

\begin{enumerate}[(a)]\label{wt of w_0}
    \item \textbf{Type $A_n$}: here
    \begin{equation*}
     w_0=
    \begin{cases}
     s_{\a_1+\cdots+\a_{2k}}s_{\a_2+\cdots+\a_{2k-1}}\cdots s_{\a_k+\a_{k+1}}, & \text{if $n=2k$;}\\
     s_{\a_1+\cdots+\a_{2k+1}}s_{\a_2+\cdots+{\a_{2k}}}\cdots s_{\a_{k+1}}, & \text{if $n=2k+1$.}
    \end{cases}    
    \end{equation*}
    Hence, 
    \begin{equation*}
    \wt(w_0)=
    \begin{cases}
    \a_1^\vee+2\a_2^\vee+\cdots+(k-1)\a_{k-1}^\vee+k\a_k^\vee+k\a_{k+1}^\vee+(k-1)\a_{k+2}^\vee+\cdots+\a_{2k}^\vee, & \text{if $n=2k$;}\\
     \a_1^\vee+2\a_2^\vee+\cdots+k\a_k^\vee+(k+1)\a_{k+1}^\vee+k\a_{k+2}^\vee+\cdots+\a_{2k+1}^\vee, & \text{if $n=2k+1$.}
\end{cases}
\end{equation*}
    \item \textbf{Type $B_n$}: here
    \begin{equation*}
     w_0=
    \begin{cases}
     s_{\a_1+2\a_2+\cdots+2\a_{2k}}s_{\a_1}s_{\a_3+2\a_4+\cdots+2\a_{2k}}s_{\a_3}\cdots s_{\a_{2k-1}+2\a_{2k}}s_{\a_{2k-1}}, & \text{if $n=2k$;}\\
     s_{\a_1+2\a_2+\cdots+2\a_{2k+1}}s_{\a_1}s_{\a_3+2\a_4+\cdots+2\a_{2k+1}}s_{\a_3}\cdots s_{\a_{2k-1}+2\a_{2k}+2\a_{2k+1}}s_{\a_{2k+1}}, & \text{if $n=2k+1$.}
    \end{cases}    
    \end{equation*}
     Hence, 
    \begin{equation*}
    \wt(w_0)=
    \begin{cases}
    2\a_1^\vee + 2\a_2^\vee + 4\a_3^\vee + 4\a_4^\vee + \cdots + 2(k-1)\a_{2k-3}^\vee + 2(k-1)\a_{2k-2}^\vee\\ + 2k\a_{2k-1}^\vee + k\a_{2k}^\vee, & \text{if $n=2k$;}\\
     2\a_1^\vee+2\a_2^\vee+4\a_3^
     \vee+4\a_4^\vee+\cdots+2k\a_{2k-1}^\vee+2k\a_{2k}^\vee+(k+1)\a_{2k+1}^\vee, & \text{if $n=2k+1$.}
\end{cases}
\end{equation*}
    \item \textbf{Type $C_n$}: here $w_0=s_{2\a_1+\cdots+2\a_{n-1}+\a_n}s_{2\a_2+\cdots+2\a_{n-1}+\a_n}\cdots s_{2\a_{n-1}+\a_n}s_{\a_n}$, and thus 
    \begin{equation*}
        \wt(w_0)=\a_1^\vee+2\a_2^\vee+\cdots+n\a_n^\vee.
    \end{equation*}
    \item \textbf{Type $D_n$}: 
     here
    \begin{equation*}
     w_0=
    \begin{cases}
     s_{\a_1+2\a_2+\cdots+2\a_{2k-2}+\a_{2k-1}+\a_{2k}}s_{\a_1}s_{\a_3+2\a_4+\cdots+2\a_{2k-2}+\a_{2k-1}+\a_{2k}}\\ s_{\a_3}\cdots s_{\a_{2k-3}+2\a_{2k-2}+\a_{2k-1}+\a_{2k}}s_{\a_{2k-3}}s_{\a_{2k}}, & \text{if $n=2k$;}\\
     s_{\a_1+2\a_2+\cdots+2\a_{2k-2}+\a_{2k-1}+\a_{2k}}s_{\a_1}s_{\a_3+2\a_4+\cdots+2\a_{2k-2}+\a_{2k-1}+\a_{2k}}\\ s_{\a_3}\cdots s_{\a_{2k-3}+2\a_{2k-2}+2\a_{2k-1}+\a_{2k}+\a_{2k+1}}s_{\a_{2k-1}+\a_{2k}+\a_{2k+1}}s_{\a_{2k-1}}, & \text{if $n=2k+1$.}
    \end{cases}    
    \end{equation*}
     Hence,
     \begin{equation*}
     \wt(w_0)=
     \begin{cases}
      2\a_1^\vee + 2\a_2^\vee + 4\a_3^\vee + 4\a_4^\vee + \cdots+ (2k-2)\a_{2k-3}^\vee + (2k-2)\a_{2k-2}^\vee + k\a_{2k-1}^\vee \\+ k\a_{2k}^\vee, & \text{if $n=2k$;}\\
      2\a_1^\vee + 2\a_2^\vee + 4\a_3^\vee + 4\a_4^\vee + \cdots+ (2k-2)\a_{2k-3}^\vee + (2k-2)\a_{2k-2}^\vee+2k\a_{2k-1}^\vee \\+k\a_{2k}^\vee + k\a_{2k+1}^\vee, & \text{if $n=2k+1$.}
     \end{cases}
     \end{equation*}

    \item \textbf{Type $E_6$}: here $w_0=s_{\a_1+2\a_2+2\a_3+3\a_4+2\a_5+\a_6}s_{\a_1+\a_3+\a_4+\a_5+\a_6}s_{\a_3+\a_4+\a_5}s_{\a_4}$, and thus
    \begin{equation*}
     \wt(w_0)=2\a_1^\vee+2\a_2^\vee+4\a_3^\vee+6\a_4^\vee+4\a_5^\vee+2\a_6^\vee.  
    \end{equation*}
   \item \textbf{Type $E_7$}: here $w_0=s_{2\a_1+2\a_2+3\a_3+4\a_4+3\a_5+2\a_6+\a_7}s_{\a_2+\a_3+2\a_4+2\a_5+2\a_6+\a_7}s_{\a_2+\a_3+2\a_4+\a_5}s_{\a_2}$\\$s_{\a_3}s_{\a_5}s_{\a_7}$, and thus
   \begin{equation*}
       \wt(w_0)=2\a_1^\vee+5\a_2^\vee+6\a_3^\vee+8 \a_4^\vee+7\a_5^\vee+4\a_6^\vee+3\a_7^\vee.
   \end{equation*}
   \item \textbf{Type $E_8$}: here $w_0= s_{2\a_1+3 \a_2+4\a_3+6\a_4+5\a_5+4\a_6+3\a_7+2\a_8}s_{2\a_1+2\a_2+3\a_3+4\a_4+
    3\a_5+2\a_6+\a_7}$\\$s_{\a_2+\a_3+2\a_4+2\a_5+2\a_6 +\a_7}s_{\a_2+\a_3+2\a_4+\a_5}s_{\a_2}s_{\a_3}s_{\a_5}s_{\a_7}$, and thus 
    \begin{equation*}
        \wt(w_0)=4\a_1^\vee + 8\a_2^\vee + 10\a_3^\vee + 14\a_4^\vee + 12\a_5^\vee + 8\a_6^\vee + 6\a_7^\vee + 2\a_8^\vee.
    \end{equation*}
    \item \textbf{Type $F_4$}: here $w_0=s_{2\a_1+3\a_2+4\a_3+2\a_4}s_{\a_2+2\a_3+2\a_4}s_{\a_2+2\a_3}s_{\a_2}$, and thus 
    \begin{equation*}
        \wt(w_0)=2\a_1^\vee+6\a_2^\vee+4\a_3^\vee+2\a_4^\vee.
    \end{equation*}
    \item \textbf{Type $G_2$}: here $w_0=s_{3\a_1+2\a_2}s_{\a_1}$, and thus \begin{equation*}
        \wt(w_0)=2\a_1^\vee+2\a_2^\vee.
    \end{equation*}
\end{enumerate}
\begin{remark}\label{cascade}
We note that the weights of longest element found above appear in a rather different context, cf. \cite{Lusz18}. In that paper, Lusztig points out that the coroots appearing as a summand of $\wt(w_0)$ as above form a so-called \emph{cascade} (terminology due to Kostant). He lists out these coroots and their sums in loc. sit. section 1.2 and section 1.8. In fact, a map $x \mapsto r_x$ is defined from the set of involutions $\CI_W$ in an irreducible Weyl group $W$ in loc. sit. via the notion of cascade, and this is crucially used in constructing certain lifts of involutions in associated reductive groups. We compare these two maps defined on the set of involutions in \cref{sec:wt=r}.
\end{remark}

\section{Covering relation in affine Weyl group}\label{sec:cover}
The goal of this section is to prove the following result.
\begin{theorem}\label{cover}
Suppose that $W$ is an irreducible finite Weyl group. Assume that
\begin{equation}\label{cover-depth}
depth(\lambda) \geq 
\begin{cases}
3, & \text{if $W \text{is of simply laced type}$;}\\
4, & \text{if $W \text{is of non-simply laced type but not of type}~G_2$;}\\
6, & \text{if $W \text{is of type}~G_2$.}
\end{cases}
\end{equation} 
Then $w:=t^{\lambda}y \gtrdot w':=t^{m\alpha^{\vee}}s_{\alpha} w$ if and only if one of the following holds.
\begin{enumerate}
    \item $m=1$ and $\ell(s_{\alpha})=\langle 2\rho, \alpha^{\vee} \rangle -1$; in this case, $w'=s_{\alpha}t^{\lambda-\alpha^{\vee}}y$.
    \item $m=\langle \alpha, \lambda \rangle$ and $\ell(s_{\alpha}y)=\ell(y)+1$; in this case, $w'=t^{\lambda}s_{\alpha}y$.
    \item $m=\langle \alpha, \lambda \rangle -1$, $\ell(s_{\alpha}y)=\ell(y)-\langle 2\rho, \alpha^{\vee} \rangle +1$; in this case, $w'=t^{\lambda-\alpha^{\vee}}s_{\alpha}y$.
\end{enumerate}
\end{theorem}
\begin{remark}
We explain how \cref{cover} implies theorem B. Let us first note that we can again reduce to the case of irreducible factors as in \cref{reduction1}. Notice that if $\l$ is regular, $\ell(ut^\l v)=\ell(u)+\ell(t^\l v)$ by \cref{len1}. Choose reduced expressions $u=s_{i_1}\cdots s_{i_l}$ and $t^\l v=s_{j_1}\cdots s_{j_k}\z $, where $s_{i_p},s_{j_q} \in \BS$ for $1 \leq p \leq l, 1\leq q \leq k$ and $\z \in \Omega$. Then $ut^\l v=s_{i_1}\cdots s_{i_l}s_{j_1}\cdots s_{j_k}\z$ is a reduced expression, and a cocover of $ut^\l v$ must be of the form of either $s_{i_1}\cdots \widehat{s_{i_m}} \cdots s_{i_l}s_{j_1}\cdots s_{j_k}\z$ for some $p \in [1,l]$ or $s_{i_1} \cdots s_{i_l}s_{j_1}\cdots \widehat{s_{j_n}} \cdots s_{j_k}\z$ for some $q\in [1,k]$. In other words, any cocover of $ut^\l v$ is obtained by 
\begin{itemize}
    \item (C1) either multiplying a cocover of $u$ with $t^\l v$, or
    \item (C2) multiplying $u$ with a cocover of $t^\l v$.
\end{itemize}
The procedure listed in (C1) is easy to describe. A cocover of $u$ is of the form $us_\a$ for some $\a \in \Phi^+$ such that $\ell(us_\a)=\ell(u)-1$. This corresponds to case (1) in theorem B. We claim that the three other cases there, i.e. (2)-(4), come from procedure described in (C2), and hence they correspond to those listed in \cref{cover}. Clearly, the third and fourth case in theorem B corresponds respectively to the second and third case in \cref{cover}. Finally, suppose that $w':=s_at^{\l-\a^\vee}v$ is a cocover of $w:=t^\l v$ such that $uw'$ is a cocover of $uw$. By \cref{len1}, $$\ell(uw')=\ell(us_a)+\ell(t^{\l-\a^\vee})-\ell(v)=\ell(us_a)+\<2\rho,\l-\a^\vee\>-\ell(v).$$
We use dominance of $\l-\a^\vee$ in the last equality; this is true as $\<\a_i,\l-a^\vee\>=\<\a_i,\l\>-\<\a_i,\a^\vee\> \geq 0$ for any simple root $a_i$ - since the maximum value of $\<\a_i,\a^\vee\>$ equals $2$ in every Cartan type, except for $G_2$ - in which case it equals $3$, cf. \cite[Chapter VI, Section 3, no. 1]{Bou}. Note that 
\begin{equation}\label{interim}
    \ell(us_\a) \leq \ell(u)+\ell(s_a) \leq \ell(u)+\<2\rho,\a^\vee\>-1.
\end{equation}
Since $\ell(ut^\l v)=\ell(u) +\<2\rho,\l\>-\ell(v)$, the cocover condition $uw \gtrdot uw'$ gives 
$$\ell(u)+\<2\rho,\l\>-\ell(v)-1=\ell(us_a)+\<2\rho,\l-\a^\vee\>-\ell(v) \leq \ell(u)+\<2\rho,\a^\vee\>-1+\<2\rho,\l-\a^\vee\>-\ell(v).$$
Hence we deduce that all of the inequalities in \cref{interim} must be equality, and thus we get $\ell(us_\a)=\ell(u)+\<2\rho,\a^\vee\>-1$. This is exactly the situation listed in the second item in theorem B.
\end{remark}

We will only focus on proving the necessity of the above conditions. For the sufficiency part, the argument in \cite[Section 4]{Milicevic21} toward the end of proposition $4.2$ can be applied. 

\smallskip
We divide our discussion of the proof in the following subsections. 
\subsection{Some useful inequalities}
In this subsection, we lay out an estimate of some relevant quantities that we shall use in the next subsection. We shall resume the assumption on the depth of $\l$ stated in \cref{cover} throughout our discussion after the first lemma below.
\begin{lemma}\label{bd-m}
Suppose $w \gtrdot w'$. Then we must have $w'=t^{m\alpha^{\vee}}s_{\alpha} w$ for some affine root $\alpha+m$ with $1 \leq m \leq \<\alpha,\lambda\>$.
\end{lemma}
\begin{proof}
Let $w=s_{i_1}\cdots s_{i_l}$ be a reduced expression, where $s_{i_j}\in \tS$. This describes a reduced gallery from $\ba$ to $w\ba$ via the chain of adjacent alcoves $\ba \rightarrow s_{i_1}\ba \rightarrow \cdots \rightarrow s_{i_1}\cdots s_{i_l}\ba$. Since the gallery is reduced, it cannot cross any hyperplane twice - hence all the alcoves lie in the dominant chamber. Let $r_j$ be the affine reflection with respect to the common wall between $\ba_{j-1}:=s_{i_1}\cdots s_{i_{j-1}}\ba$ and $\ba_j:=s_{i_1}\cdots s_{i_j}\ba$, that is $r_j=s_{i_1}\cdots s_{i_{j-1}}s_{i_j}s_{i_{j-1}} \cdots s_{i_1}$. 

\smallskip
Now suppose that $w'=s_{i_1}\cdots \widehat{s_{i_k}} \cdots s_{i_l}$ is a cocover; as before, let $\{\ba_j': 1 \leq j \leq l-1\}$ be the collection of alcoves describing the reduced gallery corresponding to this expression of $w'$ and let $r_j'$ be the associated affine reflections. Then we see that $r_j=r_j'$ and $\ba_j=\ba_j'$ for $1 \leq j \leq k-1$, and the rest of the gallery for $w'$ is the reflection of portion of the gallery for $w$ from $\ba_{k+1}$ onward with respect to the hyperplane corresponding to $r_k$. This is because for $j \geq k$ we have
\begin{equation*}
\ba_{j-1}'=s_{i_1}\cdots s_{i_{k-1}}s_{i_{k+1}}\cdots s_{i_j}\ba=s_{i_1}\cdots s_{i_{k-1}}s_{i_k}s_{i_{k-1}}\cdots s_{i_1}(s_{i_1}\cdots s_{i_j}\ba)=r_k(\ba_j).
\end{equation*}
In particular, $w'=r_k w$. 

\smallskip
Therefore, to create a cocover of $w$, we must reflect $w\ba$ with respect to some hyperplane $H_{\alpha-m}$ lying between $\ba$ and itself. Since the number of hyperplanes between the alcove $w\ba$ and the wall $H_\alpha$ is $\< \a,\l \>$, this restricts the possibility of $m$ to asserted values above.
\end{proof}
Now, $w'=t^{m\alpha^{\vee}}s_{\alpha} w=s_\alpha t^{\lambda-m\alpha^\vee}y$. By \cref{bd-m}, the coweight associated to the translation part of $w'$ belongs to following list of coweights:
\begin{equation}\label{list}
    \lambda-\alpha^{\vee}, \lambda-2\alpha^{\vee},\cdots,\lambda-(\langle \alpha,\lambda \rangle-1)\alpha^\vee=s_{\alpha}(\lambda-\alpha^{\vee}),\lambda-\langle \alpha, \lambda \rangle \alpha^{\vee}=s_{\alpha}(\lambda).
\end{equation}

Let us now define $$k=k_{\lambda,\alpha}:=\text{max} \{m: \lambda-m\alpha^\vee \in \overline{\CC^+},1 \leq m \leq \<\alpha,\lambda\>, m \in \BZ\}.$$ 

We now give a lower bound for $k$. 
\begin{lemma}\label{bd-k}
We have that 
\begin{equation*}
k \geq 
\begin{cases}
1, & \text{if $W \text{is of simply laced type}$;}\\
2, & \text{if $W \text{is of non-simply laced type}$.}
\end{cases}
\end{equation*} 
\end{lemma}
\begin{proof}
By the definition of $k$, there is a real number $\wtd m$ with $k \leq \wtd m < k+1$ such that $\lambda-\wtd m \alpha^\vee$ lies on at least one of the walls of $\CC^+$. It follows that for some simple root $\alpha_i$, we have $\< \alpha_i, \lambda - \wtd m \alpha^{\vee}\rangle=0$. Therefore, $\wtd m=\frac{\<\alpha_i,\lambda\>}{\<\alpha_i,\alpha^\vee\>}$ - hence giving $k=\lfloor \frac{\<\alpha_i,\lambda\>}{\<\alpha_i,\alpha^\vee\>} \rfloor$. Recalling the estimate about maximum value of $\langle \alpha^{\vee},\alpha_i \rangle$ in each Cartan type, we get the desired bound on $k$ from the depth condition on $\lambda$.
\end{proof}

\smallskip
Our next lemma estimates the length of the translation elements corresponding to the coweights in the list (6.1). This proof closely follows a part of the proof of \cite[Proposition 4.2]{Milicevic21}. We include it here for completeness' sake.
\begin{lemma}\label{bd-ell}
Suppose that for some integer $ m \in [1,\< \alpha,\lambda \>]$, $\lambda- m \alpha^\vee$ lies in the closure of a chamber different from $\CC^+$ or $\CC_{s_\alpha}$. Then $\ell(t^{\lambda- m \alpha^{\vee}})\leq \ell(t^{\lambda-k\alpha^{\vee}})$.
\end{lemma}
\begin{proof}
Following \cite{LS10}, define the function $f:\BR \rightarrow \BR_{\geq 0}$ by linearly extending the function $\wtd f: \BZ \rightarrow \BZ_{\geq 0}$ defined by 
\[\wtd f(m)=\ell(t^{\l-m\a^\vee}).\]
More precisely, $f$ is the function associated to the graph obtained by joining $\wtd f(m)$ and $\wtd f(m+1)$ by the line passing through them for every $m \in \BZ$. It is easy to see that $f$ is a convex function, cf. \cite{LS10}, proof of proposition 4; in fact, it is a piece-wise linear function, and is given by a single expression linear in $m$ as long as $\l-m\a^\vee$ is in the same chamber. Since $\l, \l-\a^\vee$ are both dominant due to the imposed depth hypothesis, we see that $f(1)=\<2\rho,\l-\a^\vee\>, f(\<\a,\l\>)=\<2\rho,\l\>$. 

We now show that $f$ is decreasing around $1$ and increasing around $\<\a,\l\>$. For example, if $m \in (0,\frac{3}{2})$ and $\a_i \in \D$ then by our earlier discussion about maximum value of $\<\a_i,\a^\vee\>$ we have
\begin{equation*}
\<\a_i,\l-m\a^\vee\>=\<\a_i,\l\>-m\<\a_i,\a^\vee\> \geq
\begin{cases}
\text{depth}(\l)-3, & \text{if $W \text{is not of type}~G_2$;}\\
\text{depth}(\l)-\frac{9}{2}, & \text{if $W \text{is of type}~G_2$.}
\end{cases}
\end{equation*}
Hence, our depth hypothesis ensures that $\<\a_i,\l-m\a^\vee\> \geq 0$. Therefore, $\l-m\a^\vee$ is in $\overline{\CC^+}$ whenever $m \in (0,\frac{3}{2})$, and thus 
\begin{equation}\label{interim1}
    f(m)=\<2\rho, \l-m\a^\vee\>=\<2\rho,\l\>-m\<2\rho,\a^\vee\>
\end{equation}
is clearly decreasing in this neighbourhood. Similarly, we note that if $m \in (\<\a,\l\>-\frac{1}{2},\<\a,\l\>+\frac{1}{2})$ and $\a_i \in \D$ we have
\begin{flalign*}
\<\a_i,s_\a(\l-m\a^\vee)\>=\<\a_i,\l\>-(\<\a,\l\>-m)\<\a_i,\a^\vee\> 
\geq  
\begin{cases}
\text{depth}(\l)-1, & \text{if $W \text{is not of type}~G_2$;}\\
\text{depth}(\l)-\frac{3}{2}, & \text{if $W \text{is of type}~G_2$.}
\end{cases}
\end{flalign*}
Again, our depth hypothesis ensures that $\<\a_i,s_\a(\l-m\a^\vee)\> > 0$. Hence $\l-m\a^\vee \in \CC_{s_\a}$ whenever $m \in (\<\a,\l\>-\frac{1}{2},\<\a,\l\>+\frac{1}{2})$, and thus
\begin{equation}\label{interim2}
    f(m)=\<2\rho,s_a(\l-m\a^\vee)\>=\<2\rho,\l\>-(\<\a,\l\>-m)\<2\rho,\a^\vee\>=\<2\rho,s_a(\l)\>+m\<2\rho,\a^\vee\>
\end{equation}
is clearly increasing in the neighbourhood.

Since the intervals $(0,\frac{3}{2})$ and $(\<\a,\l\>-\frac{1}{2},\<\a,\l\>+\frac{1}{2})$ are disjoint, by convexity of $f$ we infer the global shape of the graph of $f$. Namely, $f(m)$ steadily decreases and is defined by \cref{interim1} as long as $\l-m\a^\vee \in \overline{\CC^+}$; as $\l-m\a^\vee$ traverses through other chambers with value of $m$ increasing, it (weakly) decreases further until it reaches local minimum, then it starts increasing; and finally, once $\l-m\a^\vee$ enters $\overline{\CC_{s_\a}}$, $f(m)$ is defined by \cref{interim2} and steadily increases. Therefore, if we define $$k'=k'_{\l,\a}:=\text{min} \{m: \lambda-m\alpha^\vee \in \overline{\CC_{s_\a}},1 \leq m \leq \<\alpha,\lambda\>, m\in \BZ\},$$ we must have $f(m) \leq \text{max}\{(f(k),f(k')\}$ for all $m \in [k,k']\cap \BZ$. Since $\l-k\a^\vee=s_\a(\l-k'\a^\vee)$ by symmetry, we have $f(k)=f(k')=\ell(t^{\l-k\a^\vee})$. Hence we are done.
\end{proof}

\smallskip
Recall from \cref{bd-m} that $m \leq \<\a,\l\>$. We now define $z \in W$ to be such that $\lambda-m\alpha^\vee \in \overline{\CC_z}$. Note that if $\lambda-m\alpha^\vee$ is singular, $z$ is not uniquely specified by this condition. However, the precise choice of $z$ would be immaterial in what follows, and henceforth pick one such $z$ satisfying the above condition. Let us denote the left and right weak Bruhat order by $\prec_{\text{left}}$ and $\prec_{\text{right}}$ respectively.
\begin{lemma}\label{bd-z}
We have $z \prec_{\text{right}} s_\alpha$. As a consequence, $\ell(s_\alpha z)+\ell(z)=\ell(s_\alpha)$.
\end{lemma}
\begin{proof}

It suffices to show that $\text{Inv}(z^{-1}) \subset \text{Inv}(s_\alpha)$, because this is equivalent to $z^{-1} \prec_{\text{left}} s_\alpha$, which in turn is equivalent to the claim above.

We start by noting that for $\beta \in \text{Inv}(z^{-1})$, we have $\langle \beta, \lambda-m\alpha^{\vee} \rangle \leq 0$. This follows from rewriting $\langle \beta, \lambda-m\alpha^{\vee} \rangle = \langle z^{-1}\beta, z^{-1}(\lambda-m\alpha^{\vee})\rangle$ and noting that $z^{-1}(\lambda-m\alpha^\vee)$ is a dominant coweight and $z^{-1}\beta \in -\Phi^+$. Therefore we get $\langle \beta, \lambda \rangle \leq m\langle\beta,\alpha^\vee \rangle$. Since $\lambda$ dominant regular, $\langle \beta, \lambda \rangle>0$ and hence $\langle\beta,\alpha^\vee \rangle>0$ as well. Dividing out both side of the inequality by $\langle\beta,\alpha^\vee \rangle$, we get $m \geq \frac{\langle\beta,\lambda\rangle}{\langle\beta,\alpha^\vee\rangle}$. Combining this with \cref{bd-m}, we obtain 
    \[\langle\alpha,\lambda\rangle \geq \frac{\langle\beta,\lambda\rangle}{\langle\beta,\alpha^\vee\rangle}.\]
    
But then $\langle\beta,\alpha^\vee\rangle  \langle\alpha,\lambda\rangle \geq \langle\beta,\lambda\rangle$, whence $\<s_\alpha(\beta),\lambda\>=\<\beta-\<\beta,\alpha^\vee\>\alpha,\lambda\> \leq 0$. Since $\lambda$ is dominant regular, this gives $s_\alpha(\beta) \in -\Phi^+$, hence $\beta \in \text{Inv}(s_\alpha)$. Thus we have shown $\text{Inv}(z^{-1}) \subset \text{Inv}(s_\a)$.

\smallskip
Finally, let us note that 
\begin{equation*}
    \ell(s_\alpha z)+\ell(z)=\ell(s_\alpha)-\ell(z)+2|\text{Inv}(s_\alpha)^c\cap \text{Inv}(z^{-1})| + \ell(z)=\ell(s_\alpha)+2|\text{Inv}(s_\alpha)^c\cap \text{Inv}(z^{-1})|=\ell(s_\alpha).
\end{equation*}

This finishes the proof.
\end{proof}

\subsection{Two distinct possibilities}
In this subsection we further pin down the possible values of $z$. Recall that $z^{-1}(\lambda-m\alpha^\vee)$ is dominant; let us temporarily denote this coweight by $\mu$. Let $J \subset \BS$ be such that $W_J=\text{Stab}(\mu)$. Abbreviate $\zeta=z^{-1}y$. By standard fact about Coxeter groups, $\zeta$ has an unique factorization as $\zeta=\zeta_J \zeta^J$ with $\zeta_J\in W_J$ and $\zeta^J \in~^J W$. Note that 
\begin{equation}\label{interim3}
    w'=s_\alpha t^{\lambda-m\alpha^\vee}y=s_\alpha zt^{\mu}z^{-1}y=s_\alpha zt^{\mu}\zeta_J \zeta^J=s_\alpha z\zeta_J t^{\mu} \zeta^J.
\end{equation}

\begin{lemma}
The alcove $t^{\mu} \zeta^J\ba$ is in the dominant chamber.
\end{lemma}
\begin{proof}
Define $\bar{\omega}:=\frac{1}{n}\sum_{i=1}^n \frac{\omega_i^{\vee}}{n_i}$, where $n_i$ is the coefficient of $\alpha_i$ in the highest root $\theta$. Since $\mu+\zeta^J \bar{\omega}$ is centroid of the alcove $t^\mu \zeta^J \mathbf{a}$, we can make the following observation: \textit{$t^{\mu}\zeta^J\ba$ is in the dominant chamber if and only if $\mu+\zeta^J \bar{\omega}$ is dominant.}

\smallskip
We now show that the latter statement holds true. Note that if $\alpha_i \in J$, then $\langle \alpha_i, \mu+y^J \bar{\omega} \rangle=\langle (y^J)^{-1}\alpha_i,\bar{\omega}\rangle >0$, since by definition $s_i y^J > y^J \implies (y^J)^{-1} s_i > (y^J)^{-1}
\implies (y^J)^{-1}\alpha_i \in \Phi^+$. Now let $\alpha_i \in \Delta \setminus J$, then $\langle \alpha_i,\mu \rangle \geq 1$; since $(y^J)^{-1}\alpha_i \geq -\theta$, we have $\langle \alpha_i, \mu+y^J \bar{\omega} \rangle \geq 1 + \langle -\theta,\bar{\omega} \rangle \geq 1-1 = 0$. This proves the claim.
\end{proof}

Applying the previous lemma, we get from \cref{interim3}
\begin{equation*}
    \ell(w')=\ell(s_\alpha z\zeta_J)+\langle 2\rho, \mu \rangle -\ell(\zeta^J).
\end{equation*}

Therefore the cocover condition gives
\begin{align*}
& \langle 2\rho, \lambda \rangle -\ell(y) -1=\ell(s_\alpha z\zeta_J)+\langle 2\rho, \mu \rangle -\ell(\zeta^J).& 
\end{align*}

In other words, we get
\begin{equation}\label{interim4}
    \ell(t^\lambda)-\ell(t^{\mu})= 1+\ell(s_\alpha z \zeta_J)+\ell(y)-\ell(\zeta^J).
\end{equation}

Write $y=z \cdot z^{-1}y=z\zeta_J \zeta^J$. Note that \cref{len2} then gives
\begin{equation*}
    \ell(y)=\ell(z \cdot \z_J \z^J)=\ell(z)+\ell(\zeta_J)+\ell(\zeta^J)-2|\text{Inv}(z)\cap \text{Inv}((\zeta^J)^{-1}\zeta_J^{-1})|.
\end{equation*}

Similarly,
\begin{equation*}
    \ell(s_az\z_J)=\ell(s_\alpha z)-\ell(\zeta_J)+2|\text{Inv}(s_\alpha z)^c \cap \text{Inv}(\zeta_J^{-1})|
\end{equation*}

Hence we can rewrite the right hand side of \cref{interim4} as 
\begin{align*}
& 1+\ell(s_\alpha z)+\ell(z)+2\{|\text{Inv}(s_\alpha z)^c \cap \text{Inv}(\zeta_J^{-1})|-|\text{Inv}(z)\cap \text{Inv}((\zeta^J)^{-1}\zeta_J^{-1})|\}.&
\end{align*}

Since $\zeta_J^{-1} \prec_{\text{left}} (\zeta^J)^{-1} \zeta_J^{-1}$, we have that $\text{Inv}(\zeta_J^{-1}) \subset \text{Inv}((\zeta^J)^{-1} \zeta^{-1})$. Combining this with the fact that $|A|-|B| \leq |A\setminus B|$ for two sets $A, B$, we therefore conclude that
\begin{equation}\label{interim5}
   \ell(t^\lambda)-\ell(t^{\mu})  \leq 1 +\ell(s_\alpha z)+\ell(z)+2|\text{Inv}(s_\alpha z)^c \cap \text{Inv}(z)^c \cap \text{Inv}(\zeta_J^{-1})|.
\end{equation}
 
\begin{lemma}\label{bd-z-copy}
We have $\text{Inv}(s_\alpha z)^c \cap \text{Inv}(z)^c \cap \text{Inv}(\zeta_J^{-1}) = \emptyset$.
\end{lemma}
\begin{proof}
The argument here is similar to the proof of \cref{bd-z}. Suppose $\beta \in \text{Inv}(s_\alpha z)^c \cap \text{Inv}(z)^c \cap \text{Inv}(\zeta_J^{-1}) \subset \text{Inv}(s_\alpha z)^c \cap \text{Inv}(z)^c \cap \Phi_J^{+}$. Thus $\langle \beta, z^{-1}(\lambda-m\alpha^\vee)\rangle=0$, hence $\langle z\beta, \lambda \rangle=m\langle z\beta, \alpha^\vee \rangle$; since $z\beta$ is a positive root, both sides are positive and 
\begin{equation*}
    m=\frac{\langle z\beta, \lambda \rangle}{\langle z\beta, \alpha^\vee \rangle} \leq \langle \alpha,\lambda\rangle.
\end{equation*}
Therefore, $\<z\b, \l\>-\<z\b,\a^\vee\>\<\a,\l\> = \<z\b,s_\a(\l)\>\leq 0$. Since $\b \in \text{Inv}(s_\alpha z)^c$, we have $s_\a (z\b) \in \Phi^+$. Hence we have $\<z\b,s_\a(\l)\>=\<s_\a(z\b),\l\> \leq 0$, but that is a contradiction since $\l$ is dominant regular. This shows that the purported set must be empty and we are done.
\end{proof}
Therefore, combining \cref{interim5} with \cref{bd-z} and \cref{bd-z-copy} gives
\begin{equation}\label{interim6}
     \ell(t^\lambda)-\ell(t^{\mu})  \leq 1 +\ell(s_\alpha)
\end{equation}

Now, suppose that $z \neq 1, s_\alpha$. By \cref{bd-ell} this gives 
\begin{equation*}
    \ell(t^\lambda)-\ell(t^{\mu}) \geq \ell(t^\lambda)-\ell(t^{\lambda-k\alpha^\vee})=k\<2\rho,\alpha^\vee\>.
\end{equation*}

Combining this with \cref{interim6}, we get
\begin{equation*}
     k \< 2\rho, \a^{\vee} \> \leq 1+\ell(s_\a) \leq  \< 2\rho, \a^{\vee} \>
\end{equation*}

This gives $k=1$. For the non-simply laced types, this is a contradiction with previously established lower bound in \cref{bd-k}. Therefore in such cases, we get that $z \in \{1,s_\a\}$.

\smallskip
Before going forward, we make the following observation about a simply laced root system: 
\[\text{if $\<\beta,\alpha^\vee\>=2$ for a fixed coroot $\alpha$, then $\beta=\alpha$.} \]

This can be checked directly in type $A_n$ and $D_n$, where the positive roots are given by $\{e_i-e_j:1\leq i<j \leq n\}$ and $\{e_i\pm e_j:1\leq i<j \leq n\} \cup \{e_i+e_n:1\leq i <n\}$ respectively. Since root systems of type $E_6,E_7$ arise as subsystem of type $E_8$, we just give an argument for root system of type $E_8$ to conclude. Note that for root system of type $E_8$, the positive roots are of two kinds - given by $\{\pm e_i+e_j :1\leq i < j \leq 8\}$, and $\{\frac{1}{2}(e_8+\sum_{i=1}^7(-1)^{\nu(i)}e_i)$: $\sum_{i=1}^7\nu(i) \in 2\BZ\}$. We can easily see that pairing between a root $\gamma$ from the first set with a coroot $\alpha^\vee$ corresponding to elements of either sets cannot be $2$ unless $\gamma=\alpha$. The remaining possibility is $\<\frac{1}{2}(e_8+\sum_{i=1}^7(-1)^{\nu_1(i)}e_i),\frac{1}{2}(e_8+\sum_{i=1}^7(-1)^{\nu_2(i)}e_i)\>=2$, but that gives $\<e_8+\sum_{i=1}^7(-1)^{\nu_1(i)}e_i,e_8+\sum_{i=1}^7(-1)^{\nu_2(i)}e_i\>=8$ - therefore forcing $\nu_1=\nu_2$.

\smallskip
We now resume the discussion about estimating $k$ and restrict ourselves to the simply laced types. Recall that $k=1$. This means that $\lambda-\alpha^\vee \in \overline{\CC^+}$, but $\lambda-2\alpha^\vee \notin \overline{\CC^+}$; hence there is a $\beta \in \Phi^+$ such that $\<\beta,\lambda-2\alpha^\vee\>\leq 0$. By the depth condition, this yields $3 \leq \<\beta,\lambda\> \leq 2\<\beta,\alpha^\vee\>$. Therefore $\<\beta,\alpha^\vee\>=2$, and $\<\beta,\lambda\>$ is equal to either $3$ or $4$. But this gives $\beta=\alpha$, and hence $\<\alpha,\lambda\>$ is either $3$ or $4$. If $\<\alpha,\lambda\>=3$, the relevant coweights are $\lambda-\alpha^\vee, \lambda-2\alpha^\vee=s_\alpha(\lambda-\alpha^\vee), \lambda-3\alpha^\vee=s_\alpha(\lambda)$; the first one is in $\CC^+$ by the depth condition, and the last two are in $\CC_{s_\alpha}$. If $\<\alpha,\lambda\>=4$, then the relevant coweights are $\lambda-\alpha^\vee, \lambda-2\alpha^\vee, \lambda-3\alpha^\vee=s_\alpha(\lambda-\alpha^\vee), \lambda-4\alpha^\vee=s_\alpha(\lambda)$; the first one is in $\CC^+$, the last two are in $\CC_{s_\alpha}$ and $\lambda-2\alpha^\vee$ lies on the wall $H_{\alpha}$, so it lies in both $\overline{\CC^+}$ and $\overline{\CC_{s_\alpha}}$.\\

We summarize the content of this subsection as follows.
\begin{proposition}
If $s_\alpha t^{\lambda-m\alpha^\vee}y$ is a cocover of $t^{\lambda}y$, then $\lambda-m\alpha^{\vee}$ is either in $\overline{\CC^+}$ or $\overline{\CC_{s_\alpha}}$.
\end{proposition}

\subsection{Finishing the proof}
\begin{proof}[Proof of Theorem 6.1] We deal with the two cases found above.
\begin{enumerate}
    \item When $\lambda-m\alpha^{\vee}$ is in $\overline{\CC^+}$: in this case, \cref{interim6} gives us 
    \begin{align*}
        \<2\rho,\l\> - \<2\rho,\l-m\a^\vee\> \leq 1 + \ell(s_\a) \leq \<2\rho,\a^\vee\>.
    \end{align*}
    Therefore we get $m \leq 1$; hence $m=1$ and $1+\ell(s_\a)=\<2\rho,\a^\vee\>$. This corresponds to the first case in \cref{cover}.
    \item When $\lambda-m\alpha^{\vee}$ is in $\overline{\CC_{s_\alpha}}$: in this case, \cref{interim6} gives us
    \begin{equation}\label{interim7}
         \<2\rho,\l\> - \<2\rho,s_\a(\l-m\a^\vee)\> \leq 1 + \ell(s_\a).
    \end{equation}
     Since $\ell(s_a)\leq \<2\rho,\a^\vee\>-1$, we thus get
     \begin{equation*}
         (\<\a,\l\>-m)\<2\rho, \a^\vee\> \leq  \<2\rho,\a^\vee\>.
     \end{equation*}
    Therefore we get $\<\a,\l\>-m \leq 1$. Since $m \le \<\a,\l\>$, this means that either $m=\<\a,\l\>-1$ or $m=\<\a,\l\>$. 
    
    Suppose first that it is the former case. Then substituting this value of $m$ in \cref{interim7} yields $\ell(s_\a)=\<2\rho,\a^\vee\>-1$. Note that $\l-m\a^\vee=s_\a(\l-\a^\vee)$ is regular (since $\l-\a^\vee$ is regular dominant due to the depth hypothesis) and thus $J=\emptyset$; substituting $\mu=s_\a(\l-\a^\vee), z=s_\a, \z_J=1, \z^J=s_\a y$ in \cref{interim4} therefore produces $\<2\rho,\a^\vee\>=1+\ell(y)-\ell(s_\a y)$, whence we get $\ell(s_\a y)=\<2\rho,\a^\vee\>-1-\ell(y)=\ell(s_\a)-\ell(y)$. This corresponds to the third case in \cref{cover}.
    
    Now, if it is the latter case we can carry out a similar substitution in \cref{interim4} to get $0=1+\ell(y)-\ell(s_\a y)$; \cref{interim7} does not yield any information in this case. This gives $\ell(s_a y)=\ell(y)+1$ and hence it corresponds to the second case in \cref{cover}. This completes the proof.
\end{enumerate}
\end{proof}
\begin{remark}
We know a posteriori that only the first and the last two choices from the string of coweights in \cref{list} are viable - so we must necessarily put a depth condition on $\lambda$ to ensure that $\lambda-\alpha^{\vee}$ is dominant; in that case, the first coweight is in the dominant chamber and the last two are in $\CC_{s_\a}$. In other words, the least stringent hypothesis on $\lambda$ under which one can expect to prove a result like this would be that $\lambda-\alpha^{\vee}$ is dominant. The depth condition that we impose above for the simply laced types is almost as weak an assumption as this, in the sense that we require $\lambda-\alpha^{\vee}$ to be \emph{dominant regular}. Calculations in small rank seem to provide evidence for our this speculation. In other words, we suspect that the depth hypothesis can be lowered to $2$ uniformly in all the cases.
\end{remark}
\subsection{Admissible subsets of $\wtd W$}
A useful description of the admissible set is given in \cite{HeYu20} in terms of weight of minimal length paths in the quantum Bruhat graph. This relies on the characterization of covering relation in $\wtd W$ as established in \cite[Proposition 4.2]{Milicevic21} and hence as such it necessarily brings into picture the superregularity condition, cf. \cite[Proposition 3.3]{HeYu20}. Since we can strengthen the result about covering relation, we automatically get the following improvement in the description of the admissible set.
\begin{proposition}
Suppose that $W$ is an irreducible Weyl group. Assume that 
\begin{equation*}
depth(\mu) \geq 
\begin{cases}
3, & \text{if $W \text{is of simply laced type}$;}\\
4, & \text{if $W \text{is of non-simply laced type but not of type}~G_2$.}\\
6, & \text{if $W \text{is of type}~G_2$.}
\end{cases}
\end{equation*} 
Let $\lambda$ be dominant and assume that
\begin{equation*}
 \langle \rho, \mu - \lambda \rangle < 
\begin{cases}
\lceil 
\frac{\text{depth}(\mu)-3}{2} \rceil, & \text{if $W$ is of simply laced type;}\\
\lceil 
\frac{\text{depth}(\mu)-4}{2} \rceil, & \text{if $W$ is of non-simply laced type but not of type $G_2$;}\\
\lceil \frac{\text{depth}(\mu)-6}{2} \rceil, & \text{if $W$ is of type $G_2$.}
\end{cases}
\end{equation*} 
Then $xt^\lambda y \in \Adm(\mu)$ if and only if $\text{weight}(x, y^{-1}) \leq \mu -\lambda$.
\end{proposition}
The proof is of this proposition is identical to the argument made in \cite{HeYu20}, cf. section 3.3 and proposition 3.3 in loc. sit. and hence we do not repeat it here. The additional ingredient is \cref{cover}, and the conditions on $\mu$ and $\l$ are a direct reflection of that.

\section{Dimension formula for $X(\mu,b)$}\label{sec:dim}
In this section, we show that the dimension formula provided in \cite[Theorem 6.1]{HeYu20} holds in the absence of the superregularity condition imposed there. Our argument closely follows the proof scheme established in \cite{HeYu20}. Therefore, we replace most proofs with references to the corresponding arguments made in loc. sit. and only point out how to bypass certain steps that will enable us to drop the superregularity condition.
\begin{proposition}\label{virdim}
Suppose that $\mu$ is regular. Then 
\begin{equation}\label{d_adm(mu)}
    d_{\text{Adm}(\mu)}(b)=\<\rho,\mu-\nu([b])\>-\frac{1}{2}\mathrm{def}_{\bG}(b) +\frac{1}{2}\ell(w_0)-\frac{1}{2}\min \{d_\G(x,xw_0): x\in W\}.
\end{equation}
\end{proposition}
\begin{proof}
We first show that

(i) if $xt^{\lambda}y \in \text{Adm}(\mu)$, then $xt^{\lambda+\mu'}y \in \text{Adm}(\mu+\mu')$ for any dominant $\mu'$.

We write $xt^{\lambda+\mu'}y=xt^\l y y^{-1}t^{\mu'}y$, and note that $y^{-1}t^{\mu'}y \in \text{Adm}(\mu')$ by definition. Hence the claim follows from an application of \cref{adm-add}.

Now we argue that statement (a) in the course of the proof in loc. sit. holds true in the absence of the superregularity condition. In other words, we need to show that

(ii) if $xt^\l y \in \text{Adm}(\mu)$ where $\mu$ is regular, then $\<\rho, \text{weight}(x,y^{-1})\> \leq \<\rho, \mu-\l\>$.

Let $xt^\l y \in \text{Adm}(\mu)$. Choose $\mu'$ to be superregular, e.g. $\mu'=n\rho^\vee$ for large enough $n$. Apply the statement in (i) to obtain $xt^{\l+\mu'} \in \text{Adm}(\mu+\mu')$. Now, $\mu+\mu'$ is superregular and thus statement (a) in loc. sit. applies to give us 
\begin{equation*}
    \<\rho, \text{weight}(x,y^{-1})\> \leq \<\rho, (\mu+\mu')-(\l+\mu')\>=\<\rho, \mu - \l\>.
\end{equation*}

This finishes the proof of (ii). Therefore, one can use this enhanced version of statement (a) to get the established upper bound (i.e. the expression in the right hand side of \cref{d_adm(mu)}) for $d_{\text{Adm}}(\mu)$ by resorting to the same proof method in loc. sit. 

In the remaining part of the proof, the authors construct an explicit $w \in \text{Adm}(\mu) \cap \CC_x$ for some $x \in W$ and argue that
\begin{equation}\label{d_w(b)}
    d_w(b)=\<\rho,\mu-\nu([b])\>-\frac{1}{2}\text{def}_{\textbf{G}}(b) +\frac{1}{2}\ell(w_0)-\frac{1}{2} d_\G(x,xw_0).
\end{equation}
To prove this, the authors employ the description of $\text{Adm}(\mu)$ established in proposition 3.3 in loc. sit. and hence it relies on the superregularity hypothesis. However, one can bypass this simply by choosing a different candidate. Namely, let $w=xt^{\mu}(xw_0)^{-1}$ for some $x\in W$ such that $xw_0 \geq x$. Note that section 5.6 of loc. sit. provides an explicit construction of such an element $x \in W$. Then we have $w \leq xw_0 t^{\mu} (xw_0)^{-1}$ by a combination of \cref{len1} and \cref{len3}. This is where we use that $\mu$ is regular. Hence, $w \in \text{Adm}(\mu)$. As has been shown in loc. sit., it now follows from definition that \cref{d_w(b)} holds true for this element. This completes the proof that \cref{d_adm(mu)} holds for dominant regular $\mu$.
\end{proof}

\begin{proof}[Proof of theorem C]
We note that once the formula for $d_{\text{Adm}(\mu)}(b)$ in \cref{d_adm(mu)} is established, the remaining part of \cite{HeYu20} focuses on showing $$\min\{d_\G(x,xw_0): x\in W\}=\ell_R(w_0).$$ However, this is a calculation done purely on the finite Weyl group $W$, and the superregularity condition does not appear in establishing this. 

Finally, the characterization of admissible sets is used once more in the proof of the theorem in section 6 in loc. sit. to show that $xt^{\mu}w_0x^{-1} \in \text{Adm}(\mu)$ for certain element $x \in W$ satisfying $xw_0 \geq x$. However, this follows directly as we have shown above. Hence, the superregularity hypothesis on $\mu$ can be avoided completely. This finishes the proof of the theorem.
\end{proof}
\begin{remark}
Note that since we work with only split groups in this article, we assume $\s=\text{id}$ in quoting statements from loc. sit. However, the conclusion of theorem C holds for quasi-split groups as well. The additional ingredient in handling this general case is the existence of an element $x \in W$ such that $\s(x)w_0 \geq x$. This can be obtained using \cite[Theorem 5.1]{HeYu20} in the following way. Let $\CO$ be the $\s$-conjugacy class of $w_0$ and define $\ell(\CO)=\min \{\ell_R(w):w \in \CO\}$. Then the aforementioned theorem asserts that for any finite Coxeter group $W$ and a length preserving graph automorphism on $W$, we have that
\begin{equation*}
    \ell(w_0)-\ell_R(\CO)=2\max\{\ell(x): x\leq \s(x)w_0\}.
\end{equation*}
Therefore, the existence of our required element is equivalent to the assertion that $\ell(w_0) > \ell_R(\CO)$. The latter statement is true in all types arising from quasi-split groups, cf. section 5.1 in loc. sit. 
\end{remark}
\begin{remark}
We remark that \cite[Remark 6.2]{HeYu20} gives an example where $\mu$ is singular and $W$ is of type $C_4$, and they point out that the dimension formula in theorem C fails in this case. Therefore, the hypothesis in theorem C is optimal. However, we do not know if the formula for $d_{\text{Adm}}(\mu)$ given in \cref{virdim} is valid in the absence of regularity condition on $\mu$.
\end{remark}
\section{A remark about the weight function}\label{sec:wt=r}
Finally, we would like to explore any possible connection between the function $\text{wt}$ and the notion of cascades. The goal of this section is to prove the following result.
\begin{proposition}\label{wt=r}
Let $W$ be an irreducible Weyl group. Then $\text{wt}(x)=r_x$ for any involution $x$ in $W$ if and only if $W$ is of type $A_n$. 
\end{proposition}
Here $r_x \in \BZ\Phi^\vee$ is associated to $x$ as per the construction in \cite[Section 1.8]{Lusz18}, which we recall in the next subsection.
\subsection{Definition of $r_x$}
Let us briefly summarize the setup from \cite{Lusz18}, for more details we refer to section 1.1 and 1.8 in loc. sit. Let $W$ be an irreducible Weyl group $W$ and we denote the set of involutions in it by $\CI_W$. For an element $x \in \CI_W$, set $Y_x=\{\l \in X_*(T)_\BR: x(\l)=-\l\}$ and define $\Phi_x^\vee=\Phi^\vee \cap Y_x$. One can similarly define $(X_x,\Phi_x)$ using the action of $W$ on $X^*(T)_\BR$. It is then proved in loc. sit. that $(X_x,Y_x,\Phi_x,\Phi_x^\vee)$ forms a root system. Furthermore, $\Phi_x^{\vee +}=\Phi_x^\vee \cap \Phi^{\vee +}$ is a set of positive coroots for this system. 

Then one inductively defines the following subsets of $X_x$. Namely, let $\CE_{x,1}$ be the set of maximal elements in $\Phi_x^+$ with respect to the usual dominance order. Then for $i \geq 2$, let $\CE_{x,i}$ to be the maximal elements of $$\{\a \in \Phi_x^+: \<\a, \b^\vee\>=0 \text{ for all } \b \in \CE_{x,1} \cup \cdots \cup \CE_{x,i-1}\}.$$ 

Finally, one defines $\CE^\vee_x=\bigcup_{i \geq 1} \{\a^\vee: \a \in \CE_{x,i}\}$. Using this set, the following element of $\BZ\Phi^\vee$ is defined in loc. sit. $$r_x=\sum\limits_{\b \in \CE^\vee_x} \b.$$
 
The following result gives a characterization of $r_x$ constructed in this way.

\begin{theorem}\cite[Theorem 0.2]{Lusz18}\label{r_w}
There is a unique map $\CI_W \rightarrow \BZ \Phi$, $x \mapsto r_x$ such that (i)-(iii) below hold.

(i) $r_1=0$, $r_{s_\a}=\a^\vee$ for any $\a \in \D$;

(ii) for any $x\in \CI_W$ and $\a \in \D$ such that $s_\a x \neq xs_\a$, we have $s_\a(r_x)=r_{s_\a xs_\a}$;

(iii) for any $x\in \CI_W$ and $\a\in \D$ such that $s_ax=xs_a$, we have $r_{s_\a x}=r_x+\CN\a^\vee$ where $\CN\in \{-1,0,1\}$.

If in addition $G$ is simply laced we have 
$\CN\in\{-1,1\}$. 
\end{theorem}
\subsection{The notion of $W$-depth}
We now recall some feature about a statistic defined on Coxeter groups that will be relevant for us in the next subsection. For a Coxeter system $(W,\BS)$ and a set of positive roots $\Phi^+$ in it, \cite[Section 4.6]{BB05} defines \emph{$W$-depth}\footnote{In fact, this is called \emph{depth} in \cite{BB05}, \cite{BBNW16} and \cite{PT12}; we alter the terminology here to avoid any confusion with the notion of depth of a coweight defined earlier.} of an element $\b \in \Phi^+$ by $\text{dp}(\b):=\min\{k: x\cdot \b \in -\Phi^+ \text{ for some }x\in W \text{ with } \ell(x)=k\}.$

It is a classically known fact that one can use this to give a partial order on the set of roots, cf. \cite{BB05}. In \cite{PT12}, Petersen and Tenner extend this concept to define the following function on $W$ (still denoted by dp)
\[\text{dp}(x):=\min\{\sum\limits_{i}\text{dp}(\b_i): x=s_{\b_1}\cdots s_{\b_k}, \b_i \in \Phi^+\}.\]

It is easy to see that $\text{dp}(s_{\b})=\text{dp}(\b)=\frac{1}{2}(\ell(s_{\b})+1)$ for any positive root $\b$. Hence we have that $\text{dp}(s_\b) \leq \<\rho,\b^\vee\>$, and equality occurs if and only if $\b$ is a quantum root. In particular, if $x$ is an element of a Weyl group of simply laced type, we have 
\begin{equation}\label{dp}
    \text{dp}(x)=\min\{\sum\limits_{i}\<\rho,\b_i^\vee\>: x=s_{\b_1}\cdots s_{\b_k}, \b_i \in \Phi^+\}.
\end{equation}

Following \cite{BBNW16}, we say that \emph{the $W$-depth of $x$ is realized by a reduced factorization of $x$} if there exists an expression $x=s_{\b_1}\cdots s_{\b_k}$ with $\b_i \in \Phi^+$, such that $\ell(x)=\sum\limits_{i=1}^k \ell(s_{\b_i})$ and $\text{dp}(x)=\sum\limits_{i=1}^k \text{dp}(s_{\b_i})$.

We also recall the notion of \emph{reduced reflection length}\footnote{In the same vein, we alter the terminology $\ell_R$ used in loc. sit. for the definition of this length function and call it $\ell_{\text{red}}$.} $\ell_{\text{red}}$ introduced in loc. sit. Essentially, its definition is parallel to \cref{RQRD}, without the quantum root proviso - i.e. it is defined using a decomposition that satisfies only the last two conditions in that definition.

It is proved in \cite{PT12} that the $W$-depth of every element in a classical finite Coxeter group is realized by reduced factorization. This observation plays a central role in the following result.
\begin{proposition}\cite[Proposition 6.6]{BBNW16}\label{red-dp}
If the $W$-depth of $x$ is realized by a reduced factorization then $\text{dp}(x)=\frac{1}{2}(\ell(x)+\ell_{\text{red}}(x))$; in particular, this equality holds whenever $W$ is a classical finite Coxeter group.
\end{proposition}
\subsection{Relation between wt and dp}
\begin{lemma}\label{wt-sum-dp}
Let $W$ be an irreducible Weyl group of type $A_n$ or $D_n$. Then we have $\<\rho,\text{wt}(x)\>=\text{dp}(x)$ for any element $x \in W$.
\end{lemma}
\begin{proof}
Note that for any irreducible Weyl group $W$ and for any two elements $x,y \in W$, the following holds true by \cite[Section 4.2, Statement (a)]{HeYu20}:
\[\ell(y)=\ell(x)-\<2\rho,\text{wt}(x,y)\>+d_{\G}(x,y).\]
Letting $y=1$, we get 
\begin{equation}\label{eqn-wt-sum}
    \<\rho,\text{wt}(x)\>=\frac{1}{2}(\ell(x)+d_{\G}(x,1))=\frac{1}{2}(\ell(x)+\ell_\downarrow(x))
\end{equation}
Since all roots are quantum in a group of simply laced type, we have $\ell_\downarrow(x)=\ell_{\text{red}}(x)$. Now, we get the desired conclusion appealing to \cref{red-dp}. 
\end{proof}
\begin{remark}
Note that in general we have $\<\rho,\text{wt}(x)\>\geq\text{dp}(x)$ just by appealing to the definition. If $W$ is an irreducible Weyl group of non-simply laced type, the existence of an element $x \in W$ with $\ell_\downarrow(x) > \ell_{\text{red}}$ ensures that strict inequality do occur. We refer to \cref{counter2} for explicit example of such elements. As noted in \cite[Question 6.5]{BBNW16}, it is hard to verify by a computer whether the $W$-depth of all elements are realized by reduced factorization, even for a group of type $E_7$. Hence it is not clear if \cref{wt-sum-dp} holds true for the groups of type $E_6, E_7, E_8$ as well.
\end{remark}

\subsection{Weight of an involution}

It is classically known that if $x$ is an involution in an irreducible Weyl group of rank $n$, then it can be written as a product of at most $\ell_R(x)$ commuting reflections, where $\ell_R(x) \leq n$. This fact is needed in the statement of \cref{perp-wt} below. We also need the following lemma.
\begin{lemma}\label{Perm}\cite[Lemma 3.3]{Rapoport05}
Let $P$ be a $W$-stable convex polygon in $X_*(T)_{\BR}$. Let $v \in \ba$ and assume that $w_1, w_2$ are two elements of $\wtd W$ such that $w_1 \leq w_2$. Then we have $w_1v-v \in P$, if $w_2v-v \in P$.
\end{lemma} 
In our case, we are going to let $P$ to be $P_\z:=\text{Conv} (W \z)$ for certain dominant $\z$; this is the convex hull of points in its $W$-orbit. Recall that if $\theta=\sum\limits_{i=1}^n m_i\a_i$, then the vertices of $\ba$ are given by $\{\frac{\varpi^\vee}{m_i}: 1\leq i \leq n\}$, where $\{\varpi_i^\vee: 1\leq i\leq n\}$ is the set of fundamental coweights. 
\begin{proposition}\label{perp-wt}
Suppose that $W$ is a Weyl group of type $A_n$. Let $x\in W$ be an involution. Suppose that $x=s_{\b_1}\cdots s_{\b_k}$, where $k\leq n$ and $\{\b_i: 1\leq i \leq k\}$ is a set of orthogonal positive roots. Then \[\text{wt}(x)=\sum\limits_{i=1}^k \b_i^\vee.\] 
\end{proposition}
\begin{proof}
Let us first note that it suffices to show $\text{wt}(x)\geq \sum\limits_{i=1}^k \b_i^\vee$. Indeed, we then have \[\<\rho, \text{wt}(x)\> \geq \<\rho, \sum\limits_{i=1}^k \b_i^\vee\>.\] By \cref{wt-sum-dp} and \cref{dp}, the above  must be an equality. Therefore, writing $\text{wt}(x) - \sum\limits_{i=1}^k \b_i^\vee = \sum\limits_{\a \in \D} n_\a \a^\vee$ with $n_\a \in \BZ_{\geq 0}$, we get $\sum\limits_{\a \in \D} n_\a=0$, thereby showing that $n_\a=0$ for each $\a \in \D$. Hence we have the desired equality.

\smallskip
Let us choose $\l \in X_*(T)$ such that $\text{depth}(\l) \geq 2n$. By \cref{easy >}, we then have $t^\l x \geq  t^{\l-\text{wt}(x)}$. An easy computation shows that $$t^\l x \varpi_j^\vee-\varpi_j^\vee=\l-\sum\limits_{i=1}^k \<\b_i, \varpi_j^\vee\>\b_i^\vee=\l -\sum\limits_{\b_i \geq \a_j}\b_i^\vee.$$
We set $\z_j:=\l -\sum\limits_{\b_i \geq \a_j}\b_i^\vee$. Note that at most $n$ coroots are subtracted in the expression of $\z_j$, and since the maximum value of $\<\a,\b_i^\vee\>$ is $2$ for any simple root $\a$, the depth hypothesis on $\l$ ensures that $\z_j$ is dominant. Applying \cref{Perm}, we get that $\l-\text{wt}(x)\in P_{\z_j}$. Since $\l-\text{wt}(x)$ is dominant, we conclude that $\l-\text{wt}(x) \leq \z_j$, i.e.
\begin{equation}\label{join}
    \text{wt}(x) \geq \sum\limits_{\b_i \geq \a_j}\b_i^\vee.
\end{equation}
Since \cref{join} is valid for any $j=1,\cdots,n$, we get that $$\text{wt}(x) \geq \bigvee\limits_{j=1}^n (\sum\limits_{\b_i \geq \a_j}\b_i^\vee),$$ where $\bigvee$ stands for the join operation. It is easy to see that that this join is equal to $\sum\limits_{i=1}^k\b_i^\vee$. Hence we are done.
\end{proof}

\begin{remark}\label{counter1}
We note that \cite[Example 2.12]{BBNW16} gives an example of such an element $x$ where the conclusions of \cref{perp-wt} and \cref{wt=r} fail. Namely, let $W$ be a Weyl group of type $D_4$. Then $$x=s_4s_2s_3s_1s_2s_4s_2$$ is an involution such that $\ell_R(x)=3$. The only way to write $x$ as a product of $3$ commuting reflections is $x=s_{\a_1+\a_2+\a_4}s_{\a_2+\a_3+\a_4}s_2$, and thus by \cite[Section 1.8, statement (a)]{Lusz18} we conclude that $r_x$ is the sum of coroots corresponding to these roots. Hence if $\text{wt}(x)=r_x$, we would get $\text{dp}(x)=7$ by \cref{wt-sum-dp}. However, \cite[Theorem 2.9]{BBNW16} shows that $\text{dp}(x)=6$.

Since the root systems of $E_n$ for $n=6,7,8$ all contain a subsystem isomorphic to that of $D_4$, the conclusions of \cref{perp-wt} and \cref{wt=r} fail in these cases as well. By our discussion about the computation of $r_x$, we also see that $\text{wt}(x) \neq r_x$.
\end{remark}

\begin{remark}\label{counter2}
We now give some explicit examples of elements in irreducible Weyl groups of non-simply laced type, for which the conclusions of \cref{wt=r} and \cref{wt-sum-dp} fail. 

\smallskip
(i) Let $W$ be a Weyl group of type $C_3$ and let $x$ be the reflection element in $W$ corresponding to a non-quantum root $\b=\a_1+2\a_2+\a_3$, i.e. $x=s_2s_3s_1s_2s_3s_1s_2$. We claim that 

(a) $x=s_{2\a_2+\a_3}s_1s_{2\a_2+\a_3}$ is a reduced quantum reflection decomposition of $x$. 

Indeed, this is a length additive decomposition using only quantum roots. We note that $\ell_\downarrow(x)\neq 2$ since the right hand side of \cref{eqn-wt-sum} must be an integer; since $\b$ is a non-quantum root, we cannot have $\ell_\downarrow(x)=1$. Hence, $\ell_\downarrow(x)=3>1=\ell_{\text{red}}(x)$. However, note that in this case $\wt(x)=(\a_2^\vee+\a_3^\vee)+\a_1^\vee+(\a_2^\vee+\a_3^\vee)$, which does match with $r_x=\b^\vee=\a_1^\vee+2\a_2^\vee+2\a_3^\vee$.

Now, considering $W$ to be the Weyl group associated to a root system of type $B_3$, let $x'$ be the reflection element corresponding to a non-quantum root $\a_1+\a_2+\a_3$. Direct computation shows that there is only one way to write $x'$ as a product of $3$ reflections in a length additive manner: $x'=s_1s_{\a_2+\a_3}s_1$, but this uses a non-quantum root $\a_2+\a_3$. Hence, its reduced quantum reflection decomposition is $x=s_1s_2s_3s_2s_1$, thereby giving $\ell_\downarrow(x)=5>1=\ell_{\text{red}}(x)$

\smallskip
(ii) Now, let $W$ be a Weyl group of type $B_4$ and pick the element $y=s_4s_3s_4s_2s_3s_4s_1s_2s_3s_4s_2
$. We claim that

(b) $y=s_{\a_3+2\a_4}s_2s_3s_1s_{\a_2+\a_3+2\a_4}$ is a reduced quantum reflection decomposition of $y$.

Indeed, this is a length additive decomposition using only quantum roots. Direct computation shows that there are only two ways to express $y$ as a product of $3$ reflections; namely, we have $y=s_{\a_1+\a_2+\a_3+2\a_4}s_{\a_2+2\a_3+2\a_4}s_2$ and $y=s_{\a_1+\a_2+\a_3+2\a_4}s_{\a_2+\a_3+\a_4}s_{\a_3+\a_4}$, but none of these are length additive decomposition. This rules out the possibility $\ell_\downarrow(y)=3$. Since $y$ is not a reflection, we deduce that the claim is true. In this case, $\text{wt}(y)=\a_1^\vee+2\a_2^\vee+3\a_3^\vee+2\a_4^\vee$. An easy computation about the -$1$-eigenspace of $y$ shows that $r_x=\a_1^\vee+3\a_2^\vee+3\a_3^\vee+2\a_4^\vee$.

Since the root system of type $B_4$ occurs as a subsystem of type $F_4$, this example works in that case as well.

(iii) Finally, for a Weyl group of type $G_2$ we let $z=s_2s_1s_2$. This is the reflection element of $W$ corresponding to a non-quantum root $\g=\a_1+\a_2$. We have $\ell_\downarrow(z)=3 >1=\ell_{\text{red}}(z)$. Also, here $\text{wt}(z)=\a_1^\vee+2\a_2^\vee$ and $r_z=\g^\vee=\a_1^\vee+3\a_2^\vee$.
\end{remark}
In view of the above remarks, we only need to justify that $\text{wt}(x)=r_x$ for any involution $x$ in a Weyl group of type $A_n$.

\begin{proof}[Proof of Proposition 9.1] We need to show that the properties (i)-(iii) in \cref{r_w} are satisfied by $\text{wt}$ as well. Note that (i) is trivially true. To check (ii), let us choose a decomposition $x=s_{\b_1}\cdots s_{\b_k}$ as described in \cref{perp-wt}. Then $s_\a x s_\a=s_{s_\a(\b_1)}\cdots s_{s_\a(\b_k)}$ is a decomposition of similar kind for $s_\a x s_\a$. By \cref{perp-wt}, we have $r_{s_\a x s_\a}=\sum\limits_{i=1}^k s_\a(\b_i^\vee) = s_\a (r_x)$. 

Finally, we argue that property (iii) is satisfied by $\text{wt}$ for any $x\in W$, where $W$ is any irreducible Weyl group. Suppose first that $s_\a x > x$. Then we can form a path from $s_\a x$ to $1$ by concatenating the edge $s_a x \rightharpoondown x$ with a path from $x$ to $1$ of shortest length that only uses downwards edges, whence $\a^\vee+\text{wt}(x) \geq \text{wt}(s_\a x)$. By \cref{monotone-wt}, we also have $\text{wt}(s_\a x) \geq \text{wt}(x)$. Combining these two inequalities, we get that $\text{wt}(s_\a x)$ is either $\text{wt}(x)$ or $\text{wt}(x)+\a^\vee$. Similarly, if $s_\a x<x$, we can show that $\text{wt}(s_\a x)$ is either $\text{wt}(x)$ or $\text{wt}(x)-\a^\vee$. 

\smallskip
This completes the proof.
\end{proof}

\end{document}